\documentclass[11pt]{article}
\usepackage{amsmath,amssymb,amsthm,eucal}
\usepackage{color}

\title{Riemannian Manifolds in Noncommutative Geometry}

\author{Steven Lord%
\thanks{email: \texttt{steven.lord@adelaide.edu.au},
\texttt{adam.rennie@anu.edu.au},
\texttt{joseph.varilly@ucr.ac.cr}}
\\[3pt]
School of Mathematical Sciences, University of Adelaide\\
Adelaide 5005, South Australia, Australia\\[6pt]
\and
Adam Rennie$^*$
\\[3pt]
Mathematical Sciences Institute,
Australian National University\\
Acton 0200, Canberra, Australia\\[6pt]
\and
Joseph C. V\'arilly$^*$
\\[3pt]
Escuela de Matem\'atica,
Universidad de Costa Rica,
San Jos\'e 2060, Costa Rica}

\date{February 2012}

\advance\topmargin by -\headheight
\advance\topmargin by -\headsep
\evensidemargin=\oddsidemargin
\makeatletter
\def\section{\@startsection{section}{1}{\z@}{-3.5ex plus -1ex minus
  -.2ex}{2.3ex plus .2ex}{\large\bf}}
\def\subsection{\@startsection{subsection}{2}{\z@}{-3.25ex plus -1ex
  minus -.2ex}{1.5ex plus .2ex}{\normalsize\bf}}
\makeatother

\numberwithin{equation}{section} 

\theoremstyle{plain} 
\newtheorem{thm}{Theorem}[section]
\newtheorem{lemma}[thm]{Lemma}
\newtheorem{prop}[thm]{Proposition}
\newtheorem{corl}[thm]{Corollary}

\theoremstyle{definition} 
\newtheorem{defn}[thm]{Definition}
\newtheorem{cond}{Condition} 

\theoremstyle{remark} 
\newtheorem{rmk}[thm]{Remark}
\newtheorem{eg}[thm]{Example}

\DeclareMathOperator{\Cliff}{{\C\ell}} 
\DeclareMathOperator{\Dom}{Dom}   
\DeclareMathOperator{\End}{End}   
\DeclareMathOperator{\Hom}{Hom}   
\DeclareMathOperator{\linspan}{span} 
\DeclareMathOperator{\Tr}{Tr}     
\DeclareMathOperator{\tr}{tr}     
\DeclareMathOperator{\tsum}{{\textstyle\sum}} 

\newcommand{\al}{\alpha}      
\newcommand{\bt}{\beta}       
\newcommand{\dl}{\delta}      
\newcommand{\eps}{\varepsilon} 
\newcommand{\Ga}{\Gamma}      
\newcommand{\ga}{\gamma}      
\newcommand{\La}{\Lambda}     
\newcommand{\la}{\lambda}     
\newcommand{\nb}{\nabla}      
\newcommand{\Om}{\Omega}      
\newcommand{\om}{\omega}      
\newcommand{\sg}{\sigma}      

\newcommand{\A}{\mathcal{A}}  
\newcommand{\Ahat}{{\widehat A}} 
\newcommand{\B}{\mathcal{B}}  
\newcommand{\C}{\mathbb{C}}   
\newcommand{\Cc}{\mathcal{C}} 
\newcommand{\cc}{\mathbf{c}}  
\newcommand{\CDA}{\mathcal{C_D(A)}} 
\newcommand{\Chat}{{\widehat C}} 
\newcommand{\Coo}{C^\infty}   
\newcommand{\D}{\mathcal{D}}  
\newcommand{\Dhat}{{\widehat\D}} 
\renewcommand{\d}{\mathrm{d}} 
\newcommand{\del}{\partial}   
\newcommand{\Dhatreg}{\langle\Dhat\rangle} 
\newcommand{\Domoo}{\Dom^\infty} 
\newcommand{\Dreg}{\langle\D\rangle} 
\newcommand{\Dtilde}{{\widetilde\D}} 
\newcommand{\E}{\mathcal{E}}  
\renewcommand{\H}{\mathcal{H}}  
\newcommand{\half}{\tfrac{1}{2}} 
\newcommand{\hatox}{\mathrel{\widehat\otimes}} 
\newcommand{\hookto}{\hookrightarrow} 
\newcommand{\I}{\mathcal{I}}  
\newcommand{\isom}{\simeq}    
\newcommand{\K}{\mathcal{K}}  
\renewcommand{\L}{\mathcal{L}} 
\newcommand{\N}{\mathbb{N}}   
\newcommand{\op}{\circ}       
\newcommand{\ovl}{\overline}  
\newcommand{\ox}{\otimes}     
\newcommand{\oxyox}{\otimes\cdots\otimes} 
\newcommand{\phihat}{\hat\phi} 
\newcommand{\Trw}{\Tr_\omega} 
\newcommand{\w}{\wedge}       
\newcommand{\x}{\times}       
\newcommand{\Z}{\mathbb{Z}}   
\newcommand{\Zz}{\mathcal{Z}} 
\newcommand{\3}{\sharp}       
\newcommand{\8}{\bullet}      
\renewcommand{\.}{\cdot}      
\renewcommand{\:}{\colon}     

\newcommand{\bbraket}[2]{\langle\!\langle#1\mathbin{|}
                          #2\rangle\!\rangle} 
\newcommand{\braket}[2]{\langle#1\mathbin{|}#2\rangle} 
\newcommand{\hideqed}{\renewcommand{\qed}{}} 
\newcommand{\reg}[1]{\langle#1\rangle} 
\newcommand{\set}[1]{\{\,#1\,\}}  
\newcommand{\stroke}{\mathbin|}     
\newcommand{\twobytwo}[4]{\begin{pmatrix}#1 & #2 \\
                               #3 & #4\end{pmatrix}} 
\newcommand{\word}[1]{\quad\mbox{#1}\quad} 

\def\pairL_#1(#2|#3){{}_{#1}(#2\stroke#3)} 
\def\pairR(#1|#2)_#3{(#1\stroke#2)_{#3}} 
\def\scal<#1|#2>{\langle#1\stroke#2\rangle} 

\begin{document}

\maketitle

\vspace{-2pc}

\begin{abstract}
We present a definition of Riemannian manifold in noncommutative
geometry. Using products of unbounded Kasparov modules, we show one
can obtain such Riemannian manifolds from noncommutative spin$^c$
manifolds; and conversely, in the presence of a spin$^c$ structure. We
also show how to obtain an analogue of Kasparov's fundamental class
for a Riemannian manifold, and the associated notion of Poincar\'e
duality. Along the way we clarify the bimodule and first-order
conditions for spectral triples.
\end{abstract}

\tableofcontents

\parskip=6pt
\parindent=0pt

\addtocontents{toc}{\vspace{-1pc}}

\section{Introduction}
\label{sec:intro}

The noncommutative geometry program of extending differential manifold
structures via the concept of a spectral triple has seen several
recent developments. Any closed Riemannian manifold $M$ endowed with a
spin$^c$ structure can be recovered from a spectral triple over its
algebra $C^\infty(M)$ of smooth functions, via the reconstruction
theorem of Connes in its several iterations
\cite{ConnesGrav,Crux,ConnesRecon}.

We present a definition of closed noncommutative Riemannian manifold.
In the absence of a spin$^c$ structure and associated Dirac operator,
our definition is modelled on the Hodge-de Rham operator acting on
differential forms. Classically, this construction relies only on a
choice of orientation and metric.

To produce examples, and to check the compatibility of our definition
with the existing definition of spin$^c$ manifolds, we show that:
\begin{enumerate}
\item[(A)]
given a spin$^c$ manifold, we can obtain a Riemannian manifold;
\item[(B)]
given a Riemannian manifold and a ``spin$^c$ structure'', we can
obtain a spin$^c$ manifold.
\end{enumerate}
These operations are mutually inverse when both are defined. 

Underlying these operations is Plymen's theorem~\cite{Plymen}, which
identifies spin$^c$ structures on a closed Riemannian manifold $M$
with Morita equivalences between the $C^*$-algebras $C(M)$, of
continuous functions on~$M$; and $\Cliff(M)$, of continuous sections
of the Clifford algebra bundle. To translate this into the language of
spectral triples, we reformulate such Morita equivalences, or
``spin$^c$ structures'', as holding between $C^\infty(M)$ and a
suitable ``smooth subalgebra'' of $\Cliff(M)$, while ensuring that the
link to $KK$-theory is preserved.

The main tool used throughout to marry modules and spectral triples is
the Kasparov product of unbounded Kasparov modules,
\cite{BaajJ,Kucerovsky}. This has been recently revisited by
Mesland~\cite{Mesland}. Since we utilise the Kasparov product to
produce our various spectral triples, we are also able to show that
the spin$^c$ and Riemannian notions of Poincar\'e duality are carried
into each other by the operations (A) and~(B). Another essential point
in these constructions is the identification of a noncommutative
analogue of Kasparov's fundamental class for a Riemannian manifold.

Several routes towards ``almost commutative'' spectral triples have
recently been taken. The thesis of Zhang~\cite{Zhang} introduces a
spectral triple over $C^\infty(M)$ in the non-spin$^c$ case, using
twisted $K$-theory to overcome the obstruction.
\'Ca\'ci\'c~\cite{Chachich} introduces spectral triples suitable for
vector bundles over~$M$, thereby extending the reconstruction theorem
to that case, at the price of weakening the ``orientability'' axiom.
Boeijink and van~Suijlekom \cite{BoeijinkS} instead formulate a
spectral triple over the Clifford algebra bundle, and relate it to the
spin$^c$ case using Kucerovsky's work, much as we do here.

By contrast, our approach to noncommutative Riemannian manifolds makes
no assumptions of commutativity. It originated in the thesis
\cite{LordThesis} of the first author (for an earlier attempt,
see~\cite{FroehlichGR}), and consists in replacing the
``noncommutative spin$^c$'' condition for spectral triples with a
``noncommutative Riemannian'' condition.

In Section~\ref{sec:mods-morita} we prepare the ground by examining
pre-Morita equivalences of hermitian modules over smooth subalgebras
of $C^*$-algebras. Section~\ref{sec:bimod-connections} develops some
tools for studying operators on finitely generated projective modules.
Here we introduce bimodule connections and characterise those
operators satisfying a first order condition on suitable bimodules.
After reviewing basic notions of spectral triples, we lay out a
bimodule twisting procedure that represents a Kasparov product between
unbounded Kasparov modules.

In Section~\ref{sec:conditions} we come to manifold structures on
spectral triples. First we formulate such triples in the spin$^c$
case, with a slight strengthening of the usual conditions
\cite{ConnesRecon,Polaris}. Next we replace the spin$^c$ condition
with a new Riemannian condition, that does not require any spin$^c$
properties. We show in detail how such a Riemannian spectral triple
represents a generalisation of Kasparov's fundamental class in
$KK$-theory.

Finally, in Section~\ref{sec:main-thms} we state and prove precise
versions of (A) and~(B). We also show how Kasparov's fundamental class
provides a translation between the spin$^c$ and Riemannian Poincar\'e
duality isomorphisms.

We shall use the following notational conventions.
\vspace{-9pt}
\begin{itemize}
\setlength{\itemsep}{0pt}
\item
Throughout, $A$ and $B$ denote separable unital $C^*$-algebras. Script 
letters $\A$ and $\B$ denote dense $*$-subalgebras $\A \subset A$ and
$\B \subset B$. Often, $\A$ will come equipped a locally convex 
topology finer than that given by the $C^*$-norm of~$A$; and 
similarly for $\B$ and~$B$. 
\item
In a $C^*$-algebra $A$, or in its dense subalgebra $\A$, the notation
$a \geq 0$ means that $a$ is a positive element of the $C^*$-algebra 
$A$; we write $a > 0$ when $a$ is a nonzero positive element of~$A$.
\item
For any algebra $\A$, its opposite algebra will be denoted $\A^\op$
with elements $a^\op$, $b^\op$, etc., satisfying
$a^\op b^\op = (ba)^\op$.
\item
We deal with two kinds of ``inner products'': hermitian pairings with
values in a $*$-algebra $\A$ or~$\B$ are written with round brackets,
like $\pairR(e|f)_\A$ or $\pairL_\B(e|f)$; while \textit{scalar}
products of vectors in a Hilbert space have angle brackets, like
$\scal<\xi|\eta>$.
\item
The standard basis of~$\C^n$ will be written $\{u_1,\dots,u_n\}$. The
same notation will be used for the standard basis of ``column
vectors'' in the right $\A$-module $\A^n$.
\item
When we discuss tensor products of Fr\'echet algebras $\A$ and $\B$,
with the projective tensor product topology, the notation $\A \ox \B$
will refer to the \textit{completed} tensor product, which is often
written as $\A \hatox \B$; thus our $\A \ox \B$ is then a Fr\'echet
space. A similar convention will be used for balanced tensor products
of topological modules.
\item
On a Riemannian manifold $(M,g)$, there is a Clifford algebra bundle
with base $M$ generated by the Clifford product on the complexified
cotangent bundle. The notation $\Cliff(M)$ will denote the (unital,
assuming $M$ to be compact) $C^*$-algebra of continuous
\textit{sections} of this bundle; its isomorphism class does not 
depend on~$g$.
\item
If $T$ is a closed operator on a Hilbert (or Banach) space $\H$, its
domain is $\Dom T$. Its \textit{smooth domain} is
$$
\Domoo T := \bigcap_{k=1}^\infty \Dom T^k .
$$
If $\D$ is a selfadjoint operator, its regularised modulus is the 
operator $\Dreg := (1 + \D^2)^{1/2}$. Note that $\Dreg - |\D|$ is 
positive and bounded, and $\Domoo \Dreg = \Domoo |\D| = \Domoo \D$.
\end{itemize}

\subsubsection*{Acknowledgements}
This work has profited from discussions with Alan Carey, Nigel Higson,
Ryszard Nest, Iain Raeburn, and Fedor Sukochev.
The second author was supported by the Statens Naturvidenskabelige
Forskningsr{\aa}d, Denmark, 
and the third author was supported by the European Commission grant
MKTD--CT--2004--509794 at the University of Warsaw and the
Vicerrector\'ia de Investigaci\'on of the Universidad de Costa Rica.
All authors received support from the Australian Research Council.

\addtocontents{toc}{\vspace{-6pt}}

\section{Hermitian modules and Morita equivalence} 
\label{sec:mods-morita}

We begin with a preliminary discussion of hermitian modules and
bimodules over dense subalgebras of unital $C^*$-algebras. Much of
Rieffel's theory of strong Morita equivalence remains true, provided
one treads carefully when invoking spectral theory. The expected
properties of finitely generated projective modules all hold, but we
must spell it out.

The theory of Morita equivalence between $C^*$-algebras is fully laid
out in the monograph~\cite{RaeburnW}, whose notations we follow
mostly. What we require are the analogous notions for certain dense
subalgebras.

\begin{defn} 
\label{df:modules} 
Let $\A$ be a  dense subalgebra of a $C^*$-algebra $A$. A right
$\A$-module $\E$ is \textit{hermitian} if it carries a pairing
$\E \x \E \to \A : (e,f) \mapsto \pairR(e|f)_\A$, also called an
$\A$-valued \textit{inner product},%
\footnote{In \cite{RaeburnW}, $\E$ is called an \textit{inner product
$\A$-module}.}
which is linear in the \textit{second} entry and satisfies
$\pairR(e|f)_\A^* = \pairR(f|e)_\A$ and thus is antilinear in the
first entry; it is also positive definite,
$\pairR(e|e)_\A \geq 0$ in~$A$ with equality if and only if $e = 0$;
and it satisfies
$$
\pairR(eb|fa)_\A = b^* \pairR(e|f)_\A \,a \word{for all} a,b \in \A.
$$
We say that $\E$ is \textit{full} if
$\pairR(\E|\E)_\A := \linspan\set{\pairR(e|f)_\A : e,f \in \E}$ is
dense in~$\A$. A hermitian $\A$-module $\E$ is \textit{projective} if
it is a direct summand of a free module $\A^N$, and is
\textit{finitely generated} if $N$ is finite and there exist
$e_1,\dots,e_N \in \E$ such that every $e \in \E$ is of the form
$e = e_1a_1 +\cdots+ e_Na_N$ for some $a_1,\dots,a_N \in \A$.

If $\E_\A$ is a right hermitian $\A$-module, the conjugate vector
space $\E^\3$ is a \textit{left} $\A$-module under the left action
$a\cdot e^\3 = (ea^*)^\3$, where $e^\3$ is just $e \in \E$ regarded as
an element of~$\E^\3$. This conjugate module $\E^\3$ also carries a
\textit{left} hermitian pairing
$\pairL_\A(\cdot|\cdot) : \E^\3 \x \E^\3 \to \A$, given by
$$
\pairL_\A(e^\3|f^\3) := \pairR(e|f)_\A \.
$$
Left hermitian pairings are linear in the \textit{first} entry and 
antilinear in the second; they obey
$\pairL_\A(e|f)^* = \pairL_\A(f|e)$ and are positive definite
over~$\A$; and they satisfy
$$
\pairL_\A(ae|bf) = a \pairL_\A(e|f) b^* \word{for all} a,b \in \A.
$$
The formula $\|e\|^2 := \| \pairR(e|e)_\A \|_A$ defines a norm
on~$\E$; if $\E$ is complete in this norm and $\A = A$ is a
$C^*$-algebra, then $\E$ is a \textit{$C^*$-module} over~$A$.
\end{defn}

Given a right $C^*$-module $E$ over a $C^*$-algebra $A$, we denote the
$C^*$-algebra of all adjointable endomorphisms and its closed
subalgebra of $A$-compact endomorphisms by $\End_A(E)$ and
$\End^0_A(E)$ respectively. The latter is the norm closure of the
algebra of finite-rank operators, spanned by
$$
\Theta_{e,f} : g \mapsto e\,\pairR(f|g)_A,
$$
for $e,f,g\in E$. The same notation is used for left $C^*$-modules, 
where the finite-rank operators now act on the right:
$g\,\Theta_{e,f} := \pairL_A(g|e)\,f$.

In~\cite{RaeburnW}, following~\cite{RieffelInd}, a \textit{Morita
equivalence bimodule} $E$ between two $C^*$-algebras $B$ and~$A$ is
introduced as a bimodule $E = {}_BE_A$ which is both a full right
$C^*$-module over~$A$ and a full left $C^*$-module over~$B$, such 
that each algebra acts by adjointable operators on the module for the 
other, and both pairings satisfy a compatibility relation: for all
$a \in A$, $b \in B$ and $e,f,g \in E$,
\begin{equation}
\pairL_B(e\,a|f) = \pairL_B(e|f\,a^*),  \qquad
\pairR(b\,e|f)_A = \pairR(e|b^*\,f)_A,  \qquad
\pairL_B(e|f)\,g = e\,\pairR(f|g)_A.
\label{eq:Morita-relns} 
\end{equation}
If we wish instead to use bimodules relating dense subalgebras of 
$C^*$-algebras, the adjointability cannot be taken for granted, but 
it can be replaced by the following norm-continuity conditions
\cite[Defn.~6.10]{RieffelInd}.

\begin{defn} 
\label{df:Morita-bimod}
Let $\A$, $\B$ be dense subalgebras of $C^*$-algebras. A
\textit{pre-Morita equivalence bimodule} $\E$ between $\B$ and $\A$ is
a $\B$-$\A$-bimodule that is both a full right hermitian $\A$-module
and a full left hermitian $\B$-module, such that for all $a \in \A$,
$b \in \B$ and $e,f,g \in \E$, the following relations hold:
\begin{equation}
\pairL_\B(e\,a|e\,a) \leq \|a\|^2 \, \pairL_\B(e|e),  \qquad
\pairR(b\,e|b\,e)_\A \leq \|b\|^2 \, \pairR(e|e)_\A,  \qquad 
\pairL_\B(e|f)\,g = e\,\pairR(f|g)_\A.
\label{eq:Morita-norms} 
\end{equation}
If $\E$ is a pre-Morita equivalence bimodule between $\B$ and~$\A$,
then $\E^\3$ is a pre-Morita equivalence bimodule between $\A$
and~$\B$.
\end{defn}

For Morita equivalence bimodules between $C^*$-algebras, the two
conditions \eqref{eq:Morita-relns} and~\eqref{eq:Morita-norms} are
equivalent: see \cite[Lemma~3.7]{RaeburnW}.

Next we show that, starting from a pre-Morita equivalence bimodule
$\E$ between \textit{unital} algebras $\B$ and~$\A$, that $\E$ is
finitely generated and projective both as an $\A$-module and as a
$\B$-module. In the case of a right $C^*$-module $E$ (not necessarily
full) over a unital $C^*$-algebra $A$, it is well known ---see, for
instance, \cite[Lemma~6.5]{KasparovTech} or
\cite[Prop.~3.9]{Polaris}--- that $E$ is a finitely generated
projective $A$-module if and only if $1_E$ is an $A$-compact
endomorphism of~$E$.

In the case that $E$ is indeed a finitely generated projective right
$A$-module, we can find elements $x_1,\dots,x_m,y_1,\dots,y_m \in E$
such that
$$
1_E = \sum_{i=1}^m \Theta_{x_i,y_i}.
$$
Then there is an idempotent $q \in M_m(A)$ and an isomorphism
$E \to q A^m$ given by
\begin{equation}
e \mapsto \bigl[ \pairR(y_i|e)_A \bigr]_i \in A^m,  \qquad
1_E \mapsto q := \bigl[ \pairR(y_i|x_j)_A \bigr]_{i,j} \in M_m(A).
\label{eq:module-iso} 
\end{equation}
These formulas also apply when $A$ and $E$ are replaced by a dense 
subalgebra $\A \subset A$ and a finitely generated projective 
hermitian right $\A$-module, $\E$. In the $C^*$-case, we can go a 
little further, see~\cite[Prop.~3.9]{Polaris}, and assume that
$y_i = x_i$ for each~$i$, so that $1_E = \sum_{i=1}^m\Theta_{x_i,x_i}$
and $q = q^*$ in~$M_m(A)$. This refinement is also available for 
dense subalgebras $\A \subset A$ that allow enough functional 
calculus to take positive square roots of positive elements.

The \emph{holomorphic} functional calculus may also be invoked to 
refine the discussion. Recall that a dense subalgebra $\A$ of a
$C^*$-algebra $A$ is called a \textit{pre-$C^*$-algebra} if for all
$a \in \A$ and all functions $f$ defined and holomorphic in a
neighbourhood of the spectrum of~$a$, we have $f(a) \in \A$. If 
$\A$ is a Fr\'echet pre-$C^*$-algebra, so also is
$M_m(\A)$~\cite{Schweitzer}.

\begin{lemma} 
\label{lm:smooth-proj} 
Suppose that $\A$ is a Fr\'echet pre-$C^*$-algebra with
$C^*$-completion $A$ and that $E$ is a right $C^*$-module over~$A$.
Suppose that $1_E$ is an $A$-compact endomorphism of~$E$. Then there
exists $m \in \N$ and a projector $q \in M_m(\A)$ such that
$\E := q\A^m$ has $E$ as its $C^*$-module completion.
\end{lemma}

\begin{proof}
The hypothesis $1_E \in \End_A^0(E)$ implies that there is an
$A$-module isomorphism $E \isom \tilde q A^m$ for some~$m$ and some
$\tilde q \in M_m(A)$. Since $\A$ is a Fr\'echet pre-$C^*$-algebra, so
also is $M_m(\A)$; thus $\tilde q$ can be norm-continuously homotopied
to a projector $q \in M_m(\A)$, \cite[pp.~21--23]{Blackadar}. In
consequence, $q = u\tilde qu^*$ for some unitary $u \in M_m(A)$.

Thus without loss of generality, $E \isom q A^m$ where the projector 
$q$ can be taken in~$M_m(\A)$. However, by regarding $M_m(\A)$ as the 
finite-rank endomorphisms of~$\A^m$, we can find column vectors
$w_1,\dots,w_m,z_1,\ldots,z_m \in \A^m$ such that
$q = \sum_{i=1}^m \Theta_{w_1,z_1} +\cdots+ \Theta_{w_m,z_m}$.
Since, moreover,
$$
q = q^2 = \sum_{i=1}^m \Theta_{qw_i,z_i}
= \sum_{i=1}^m \Theta_{w_i,qz_i} \,,
$$
we can choose $w_i,z_i \in q \A^m$. Thus $\E := q\A^m$ is a finitely
generated projective right hermitian $\A$-module, and by
\cite[Lemma~2.16]{RaeburnW} $\E$ may be completed in norm to a
right $C^*$-module $\ovl{\E}$ over~$A$. It is now routine to show that
$E \isom \ovl{\E}$ as right $A$-modules.
\end{proof}

\begin{lemma} 
\label{lm:Morita}
Let $E = {}_BE_A$ be a Morita equivalence bimodule between the unital
$C^*$-algebras $B$ and~$A$. Suppose moreover that $\B \subset B$ and
$\A \subset A$ are unital Fr\'echet pre-$C^*$-algebras. Then there are
$n,m \in \N$ and projectors $p \in M_n(\B)$ and $q \in M_m(\A)$ such
that the left hermitian $\B$-module $\E_1 := \B^n p$ and the right
hermitian $\A$-module $\E_2 := q\A^m$ both have $C^*$-module
completion isomorphic to $E$.
\end{lemma}

\begin{proof}
Since $E$ is a Morita equivalence bimodule between unital
$C^*$-algebras, there are isomorphisms $\End^0_A(E) \isom B$ and
$\End^0_B(E) \isom A$, whence the identity map $1_E$ is a compact
endomorphism for both module structures.

Using Lemma~\ref{lm:smooth-proj}, we can write $_BE \isom B^n p$ and
$E_A \isom q A^m$ for some $n,m \in \N$, $p \in M_n(\B)$ and
$q \in M_m(\A)$. We may identify $\E_1 = \B^n p$ with a 
$\B$-submodule of~$E$ and $\E_2 = q\A^m$ with an $\A$-submodule
of~$E$, under these isomorphisms. Since both module structures induce
the same norm on~$E$, given by
\begin{equation}
\|e\|^2 := \pairL_\B(e|e) = \pairR(e|e)_\A \,,
\label{eq:ambi-norm} 
\end{equation}
by \cite[Prop.~3.11]{RaeburnW}, both of these submodules are
norm-dense in~$E$; thus the completions of $\B^n p$ and $q\A^m$ are
each isomorphic to~$E$.
\end{proof}

\begin{prop} 
\label{pr:pre-Morita}
Let $\E$ be a pre-Morita equivalence bimodule between the unital
pre-$C^*$-algebras $\B$ and~$\A$. Then there are $n,m \in \N$ and
projectors $p \in M_n(\B)$ and $q \in M_m(\A)$ that define algebra
isomorphisms $\A \isom p M_n(\B) p$ and $\B \isom q M_m(\A) q$; and
module isomorphisms $\E_\A \isom q\A^m$ and ${}_\B\E \isom \B^n p$.
\end{prop}

\begin{proof}
The existence of a pre-Morita equivalence bimodule entails that the
pre-$C^*$-algebras $\B$ and~$\A$ have well-defined completions to
$C^*$-algebras, $B$ and~$A$ respectively. (Equivalently: there is a
unique continuous $C^*$-norm on~$\A$ which defines $A$ by completion
in this norm; and likewise for~$\B$.) The $C^*$-module completion $E$
of~$\E$ in the norm~\eqref{eq:ambi-norm} is a Morita equivalence
bimodule between the $C^*$-algebras $B$ and~$A$: this follows from the
continuity conditions \eqref{eq:Morita-norms} on the inner products
\cite[Prop.~3.12]{RaeburnW}. Thus $E$ is finitely generated and
projective both as an $A$-module and as a $B$-module, by Lemma
\ref{lm:Morita}. Moreover we can choose the projectors $p$, $q$
describing these module structures to lie over $\B$ and $\A$
respectively.

The vector space $\I = \pairL_\B(\E|\E)$ is a dense ideal in~$\B$,
since $\E$ is $\B$-full by hypothesis, and it is also contained in 
$\pairL_B(E|E)$, which is thereby a dense ideal of~$B$. (Recall that
$\B \hookto B$ is a continuous dense inclusion.) Therefore,
$\pairL_B(E|E) = B$ since the unital $C^*$-algebra $B$ can not have a
proper dense ideal. But the pre-$C^*$-subalgebra $\B$ of~$B$ also has
this ``good'' property, since $\set{b \in \B : \|1 - b\|_B < 1}$ is an
open neighbourhood of~$1$ in~$\B$ from continuity of the inclusion
$\B \hookto B$, and therefore the dense ideal $\I$ of~$\B$ cannot be
proper either.

Therefore $\pairL_\B(\E|\E) = \B$; and by the same token,
$\pairR(\E|\E)_\A = \A$. Thus we can write
$$
1 = \sum_{i=1}^m \pairL_\B(x_i|y_i) \in \B,  \qquad
1 = \sum_{k=1}^n \pairR(w_k|z_k)_\A  \in \A,
$$
for some $x_i,y_i,w_k,z_k \in \E$. Therefore, the maps 
$e \mapsto \bigl[ \pairL_\B(e|w_k) \bigr]_k$
and $e \mapsto \bigl[ \pairR(y_i|e)_A \bigr]_i$
described in \eqref{eq:module-iso} yield isomorphisms
$\E \isom \B^n p$ of left $\B$-modules and 
$\E \isom q\A^m$ of right $\A$-modules where
$$
p := \bigl[ \pairL_\B(z_l|w_k) \bigr]_{l,k} \in M_n(\B),  \qquad
q := \bigl[ \pairR(y_i|x_j)_\A \bigr]_{i,j} \in M_m(\A).
$$
The $*$-algebra isomorphisms $\A \isom p M_n(\B) p$ and
$\B \isom q M_m(\A) q$ are now routine. 
\end{proof}

\begin{lemma} 
\label{lm:matr-pairing}
Let $\A$ be a subalgebra of a unital $C^*$-algebra $A$ with $1 \in \A$
and let $\E = q\A^m$ a finitely generated projective right
$\A$-module. Then every Hermitian pairing on $\E$ is of the form
\begin{equation}
\pairR(e|f)_\A = \pairR(e|f)_r \equiv \tsum_{j,k} e^*_j r_{jk} f_k
\label{eq:matr-pairing} 
\end{equation}
where $e = (e_1,\dots,e_m)^T$ with each $e_j \in \A$ and $qe = e$;
similarly for $f$; and $r = [r_{jk}] \in qM_m(\A)q$ is positive.
\end{lemma}

\begin{proof}
Write $x_j = qu_j = \sum_k q_{kj} u_k \in q\A^m$, where 
$\{u_1,\dots,u_m\}$ is the standard basis of~$\A^m$. These $x_j$ 
generate $\E$ as a right $\A$-module: $e = \sum_j x_j e_j$ for any
$e \in \E$.

If $\pairR(\cdot|\cdot)_\A$ is a hermitian pairing on~$\E$, then
$0 \leq \pairR(e|e)_\A = \sum_{j,k} e_j^* \pairR(x_j|x_k)_\A e_k$.
Hence the matrix $r := [\pairR(x_j|x_k)_\A]_{jk} \in M_m(\A)$ is
positive in $M_m(A)$, by \cite[Proposition 1.20]{Polaris}. Next
\begin{align*}
(qr)_{ik} &= \tsum_j q_{ij} \pairR(x_j|x_k)_\A
= \tsum_j \pairR(x_j q_{ji}|x_k)_\A
= \tsum_{j,l} \pairR(q_{lj} u_l q_{ji}|x_k)_\A
\\
&= \tsum_{j,l} \pairR(q_{lj} q_{ji} u_l|x_k)_\A
= \tsum_l \pairR(q_{li} u_l |x_k)_\A
= \pairR(x_i|x_k)_\A = r_{ik},
\end{align*}
and similarly $rq = r$.
\end{proof}

\begin{lemma} 
\label{lm:smooth-pairing}
Let $\A$ be a dense pre-$C^*$-subalgebra of the unital $C^*$-algebra
$A$ with $1 \in \A$. If $q \in M_m(\A)$ is a projector and if 
$\pairR(\cdot|\cdot)_\A$ is an $\A$-valued hermitian pairing on
$\E = q\A^m$ making $\E$~full, then it coincides with the pairing
\eqref{eq:matr-pairing} for some positive invertible
$r \in qM_m(\A)q$.
\end{lemma}

\begin{proof}
By \cite[Lemma 2.16]{RaeburnW}, the Hermitian form
$\pairR(\cdot|\cdot)_\A$ on $\E$ has a canonical extension to an
$A$-valued pairing on the completion $E = qA^m$, which is a full right
$A$-module. By Lemma~\ref{lm:matr-pairing}, the original inner product
and this extension are both given by the
formula~\eqref{eq:matr-pairing}, for some positive element
$r \in q M_m(\A) q$.

The compact endomorphisms of the full right $A$-module $E$ are given 
by $\End_A^0(E) = q M_m(A) q$; this algebra is generated by the
rank-one operators
$\Theta^r_{e,f} \: E \to E : g \mapsto e \pairR(f|g)_r$. In terms of 
the standard pairing on $E = q A^m$ given by 
$\pairR(e|f)_A := \sum_{j,k} e^*_j q_{jk} f_k$, it follows that
$\Theta^r_{e,f}(g) = \Theta_{e,f}(rg)$. If $r$ were not invertible in
$q M_m(A) q$, the operators $\Theta^r_{e,f} = \Theta_{e,f}\,r$ and
their adjoints would generate a proper two-sided ideal of this
$C^*$-algebra, contradicting fullness of~$E$.

We remark in passing that $r$ is invertible in $q M_m(A) q$ if and 
only if $r + (1 - q)$ is invertible in $M_m(A)$, with inverse
$r^{-1} + (1 - q)$, as is easily checked. Using the stability under
the holomorphic functional calculus of $M_m(\A)$, we find that
$r^{-1}$ also lies in $qM_m(\A)q$.
\end{proof}


Later on, we will consider Hilbert spaces arising as completions of 
finitely generated projective modules.

\begin{defn} 
\label{df:ltwo-space}
Let $\A$ be a unital pre-$C^*$-algebra, $\E$ a right hermitian
$\A$-module and $\psi\: \A \to \C$ a faithful bounded positive linear
functional. Let $L^2(\E,\psi)$ denote the Hilbert space completion
of~$\E$ with respect to the scalar product
$$
\scal<e|f> := \psi\bigl( \pairR(e|f)_\A \bigr).
$$
\end{defn}

\begin{prop} 
\label{pr:lin-bdd}
Let $\H_\infty \subset \H$ be a dense subspace of the Hilbert space
$\H$, and suppose that $\H_\infty$ is a finite projective right hermitian 
$\A$-module, $\H_\infty \isom q\A^m$. Suppose moreover that
$\psi\: \A \to \C$ is a faithful bounded positive linear functional
such that $\H = L^2(\H_\infty,\psi)$.

Let $T\: \H_\infty \to \H_\infty$ be right $\A$-linear. Then $T$
extends to a bounded operator on~$\H$.
\end{prop}

\begin{rmk} 
Of course Proposition~\ref{pr:lin-bdd} applies equally well to left
modules when the operator $T$ is linear over the left action of~$\A$.
\end{rmk}

\begin{proof}
The projective right $\A$-module $\H_\infty \isom q\A^m$ has a finite
set of generators $\xi_1,\dots,\xi_m$ such that any
$\xi \in \H_\infty$ may be written as
$$
\xi = \sum_{j=1}^m \xi_j a_j = \sum_{j,r=1}^m \xi_j q_{jr} a_r,
\word{for some} a_1,\dots,a_m \in \A.
$$
The generators may thus be taken to satisfy the relations
$\xi_r = \sum_{j=1}^m \xi_j q_{jr}$, for $r = 1,\dots,m$. If
$\eta = \sum_{j=1}^m \xi_j b_j$ also, the $\A$-valued inner product of
$\eta,\xi \in \H_\infty$ coming from this isomorphism is well defined
by
$$
\pairR(\eta|\xi)_\A := \sum_{j,r=1}^m b_j^* q_{jr} a_r.
$$

The right $\A$-linearity of $T$ gives
$T\xi = \sum_{j=1}^m (T\xi_j) a_j$ and 
$T\xi_r = \sum_{j=1}^m (T\xi_j) q_{jr}$. Thus
$$
\pairR(T\xi_j|T\xi_k)_\A 
= \sum_{r,s} q_{jr}\, \pairR(T\xi_r|T\xi_s)_\A \,q_{sk}.
$$
The inner product $\pairR(T\xi|T\xi)_\A$ may be expanded as follows:
\begin{align*}
\pairR(T\xi|T\xi)_\A
&= \sum_{j,k} a_j^* \pairR(T\xi_j|T\xi_k)_\A \, a_k
= \sum_{j,k,r,s} a_j^* q_{jr}\, \pairR(T\xi_r|T\xi_s)_\A \,q_{sk} a_k
\\
&= \sum_{r,s} \pairR(\xi|\xi_r)_\A \pairR(T\xi_r|T\xi_s)_\A
\pairR(\xi_s|\xi)_\A
= \sum_{r,s} \bigl( T\xi_r \,\pairR(\xi_r|\xi)_\A \bigm|
T\xi_s \,\pairR(\xi_s|\xi)_\A \bigr)_\A
\\
&= \sum_{r,s} \bigl( \Theta_{T\xi_r,\xi_r} \xi \bigm|
\Theta_{T\xi_s,\xi_s} \xi \bigr)_\A
= \sum_{r,s}  \bigl( \xi \bigm| 
\Theta_{\xi_r \pairR(T\xi_r|T\xi_s)_\A, \xi_s} \xi \bigr)_\A
\\
&\leq \sum_{r,s}
\bigl\| \Theta_{\xi_r \pairR(T\xi_r|T\xi_s)_\A, \xi_s} \bigr\|
\,\pairR(\xi|\xi)_\A \,.
\end{align*}
In the last line here the norm is the endomorphism norm (of the
$C^*$-completion), which satisfies
$\|\Theta_{\xi,\eta}\| = \| \pairR(\eta|\xi)_\A \|$ for
$\xi,\eta \in \H_\infty$ by \cite[Lemma~2.30]{RaeburnW}.

Now we can estimate the operator norm of $T$; for
$\xi \in \H_\infty$ we get the bound
\begin{align*}
\scal<T\xi|T\xi> = \psi\bigl( \pairR(T\xi|T\xi)_\A \bigr)
&\leq \sum_{r,s=1}^m \bigl\| 
\bigr( \xi_s \bigm| \xi_r \pairR(T\xi_r|T\xi_s)_\A \bigr)_\A \bigr\|
\, \psi\bigl( \pairR(\xi|\xi)_\A \bigr)
\\
&= \sum_{r,s=1}^m \bigl\| 
\bigr( \xi_s \bigm| \xi_r \pairR(T\xi_r|T\xi_s)_\A \bigr)_\A \bigr\|
\, \scal<\xi|\xi> \,.
\end{align*}
Therefore, $T$ extends by continuity to a bounded operator on~$\H$, 
with
\begin{equation*}
\|T\|^2 \leq \sum_{r,s=1}^m
\bigl\| \pairR(\xi_s|\xi_r)_\A \pairR(T\xi_r|T\xi_s)_\A \bigr\|.
\tag*{\qedhere}
\end{equation*}
\end{proof}

\addtocontents{toc}{\vspace{-6pt}}

\section{Bimodule connections and Kasparov modules} 
\label{sec:bimod-connections}

In this section we shall use the Kasparov product to produce new
spectral triples from old ones, in the presence of a Morita
equivalence bimodule. This is essentially the noncommutative
formulation of twisting an elliptic operator by a vector bundle.
Spectral triples yield unbounded Kasparov modules, so we use the work
of Kucerovsky~\cite{Kucerovsky} to implement this product. The main
technical requirement is the construction of suitable bimodule
connections to set up the unbounded Kasparov product.

\subsection{Connections on bimodules and the first order condition} 
\label{ssc:Dirac-type}

In this subsection the $*$-algebras $\A$ and $\B$ will always be
unital \textit{Fr\'echet pre-$C^*$-algebras} with a unique continuous
$C^*$-norm; and $\E$ will denote a $\B$-$\A$-bimodule with a complete
locally convex topology for which the module operations are
continuous. Moreover, we shall assume that $\E$ carries both a
left-linear inner product with values in $\B$ and a right-linear inner
product with values in $\A$, such that the right action of $\A$ on
$\E$ is $\B$-linear and adjointable with respect to
$\pairL_\B(\.|\.)$; and vice versa.

We do not, at this stage, assume any fullness conditions, so that
$\E$ need not be a pre-Morita equivalence bimodule between $\B$ 
and~$\A$, although this is of course the main example.

Let $\End_\A(\E)$ be the $*$-algebra of $\A$-linear adjointable maps
from $\E$ to~$\E$. Such maps need not, \textit{a priori}, be bounded
with respect to the $C^*$-module norm on~$\E$ coming from
$\pairR(\.|\.)_\A$; although we shall have occasion later, in 
the finitely generated and projective case, to obtain
such bounds by applying Proposition~\ref{pr:lin-bdd}. Similarly, let
$\End_\B(\E)$ be the $*$-algebra of $\B$-linear adjointable maps 
on~$\E$. Our assumptions give us the inclusions:
$$
\A \subseteq \End_\B(\E) \mbox{ acting on the right}, \qquad
\B \subseteq \End_\A(\E) \mbox{ acting on the left}.
$$
If $a \in \A$, we write $a^\op$ for the corresponding right
multiplication operator in $\End_\B(\E)$, since the right action of
$\A$ on~$\E$ gives a left action of the opposite algebra $\A^\op$.

We need to deal with tensor products over the Fr\'echet algebras $\A$ 
and~$\B$. We recall that the completed projective tensor product
(over~$\C$) of $\A$ with itself, which we write simply as $\A \ox \A$
rather than $\A \hatox \A$, is a Fr\'echet space. Since the
multiplication of a Fr\'echet algebra is jointly continuous, the 
kernel $\Om^1\A$ of the corresponding linear map
$m\: \A \ox \A \to \A$ is closed and thus is a Fr\'echet space, as
well as being an $\A$-$\A$-bimodule for which the left and right
actions of $\A$ on $\Om^1\A$ are continuous: see Section~II.4
of~\cite{Helemskii}.

Likewise, we may define \textit{balanced} tensor products such as
$\E \ox_\A \Om^1\A$, by quotienting the (completed) projective tensor
product $\E \ox \Om^1\A$ by the closure of the linear span of the
tensors $e\,a \ox \om - e \ox a\,\om$, where $e \in \E$, $a \in \A$
and $\om \in \Om^1\A$. Thus $\E \ox_\A \Om^1\A$ is again a complete
locally convex space, and moreover is a topological
$\B$-$\A$-bimodule, by~\cite[Prop.~5.15]{Helemskii}.

\vspace{6pt}

Assume further that $\E$ is projective both  as a right $\A$-module 
and as a left $\B$-module. This assumption enables us to choose
connections \cite{Book,Polaris}:
$$
\nb_\B : \E \to \Om^1\B \ox_\B \E,  \qquad\qquad
\nb_\A : \E \to \E \ox_\A \Om^1\A,
$$
which are $\C$-linear maps such that, for all $a \in \A$, $b \in \B$,
$e \in \E$,
$$
\nb_\B(be) = db \ox e + b\,\nb_\B(e),  \qquad
\nb_\A(ea) = e \ox da + \nb_\A(e)\,a.
$$
The graded differential algebra $\Om^\8\A$ is densely generated by
elements $a \in \A$, $da \in \Om^1\A$ subject to the preexisting
algebra relations of $\A$, the derivation rule
$d(ab) = a\,db + da\,b$, and the relations
$$
d(a_0 \,da_1\cdots da_k) = da_0\,da_1\cdots da_k, \qquad
d(da_1\,da_2\cdots da_k) = 0.
$$
Note that $\Om^0\A = \A$, so $\A$ is to be regarded as a subalgebra of
$\Om^\8\A$. Also, since $\A$ is a $*$-algebra, $\Om^\8\A$ becomes a
$*$-algebra by adding the rule $(da)^* = - d(a^*)$.

The connection $\nb_\A$ extends to an operator on the module
$\E \ox_\A \Om^\8\A$, using a graded Leibniz rule, and similarly
for~$\nb_\B$ on the module $\Om^\8\B \ox_\B \E$.

While $\nb_\A$ does not intertwine the left $\B$-module structures of
$\E$ and $\Om^1\B \ox_\B \E$ (and similarly for $\nb_\B$), the
following linearity relations do hold.

\begin{lemma} 
\label{lm:basic-obs}
With the notation as above, all $a \in \A$ and $b \in \B$ satisfy
$$ 
[\nb_\B, b] \in \Hom_\A(\E, \Om^1\B \ox_\B \E),  \qquad
[\nb_\A, a^\op] \in \Hom_\B(\E, \E \ox_\A \Om^1\A).
$$
\end{lemma}

\begin{proof}
This follows immediately from the definition of connection, since if
$e \in \E$ then
$$ 
\nb_\B(be) - b\,\nb_\B(e) = db \ox e.
$$
Since $\Om^1\B \ox_\B \E$ and $\E$ carry the same action of $\A$ on
the right of $\E$, the first statement is proved. 
If $a \in \A$, an identical argument shows that $[\nb_\A, a^\op]$
respects the left actions of~$\B$.
\end{proof}

Suppose now that we are given a pair of $*$-homomorphisms:
$$ 
c_\B : \Om^\8\B \to \End_\A(\E),  \qquad 
c_\A : (\Om^\8\A)^\op \to \End_\B(\E).
$$
We suppose, for compatibility with the tensor products, that
$c_\B\bigr|_{\Om^0\B}$ agrees with the left action of $\B$ on~$\E$, and
similarly for $c_\A$. With this assumption, we also obtain adjointable
module maps
\begin{equation}
\ga_\B : \Om^\8\B \ox_\B \E \to \E,  \qquad
\ga_\A : \E \ox_\A \Om^\8\A \to \E,
\label{eq:module-maps} 
\end{equation}
given by $\ga_\B(\bt \ox e) := c_\B(\bt)(e)$ and 
$\ga_\A(e \ox \al) := c_\A(\al^\op)(e)$.

\begin{eg} 
\label{eg:cliff-bundle}
The main classical example we have in mind is a Clifford module. Let
$M$ be a closed $C^\infty$ manifold with Riemannian metric~$g$, and
$E \to M$ a smooth complex vector bundle with the additional property
that each fibre $E_x$, $x \in M$, is a module for $\Cliff_x$, the
complex Clifford algebra for $T^*_xM$ with inner product given by
$g_x^{-1}$. Denoting the exterior derivative by $\d$ and the Clifford
action at $x \in M$ by~$c_x$, we can define an algebra homomorphism
$$
c_M : \Om^\8 \Coo(M) \to \End_{\Coo(M)}\bigl( \Coo(M,E) \bigr)
$$ 
by
$$
c_M(f_0\,df_1\,\cdots df_k) s
: x \longmapsto f_0(x) c_x(\d f_1)\dots c_x(\d f_k) s(x), \qquad 
f_j \in \Coo(M),\ s \in \Coo(M,E).
$$
It is straightforward to check that $\A = \B = \Coo(M)$ is a Fr\'echet
pre-$C^*$-algebra and that $\E = \Coo(M,E)$ is a $\Z_2$-graded
hermitian bimodule over this algebra with continuous actions; and that
the restriction of the left Clifford action to functions is just the
multiplication of sections by smooth functions.
\end{eg}

Returning to the general case, we note the following preliminary
result.

\begin{lemma} 
\label{lm:alg-dee} 
With the notation as above, define two linear operators on~$\E$ by 
$\D_\B := \ga_\B \circ \nb_\B$ and $\D_\A := \ga_\A \circ \nb_\A$.
Then $[\D_\B, b] \in \End_\A(\E)$ with $[\D_\B, b]^* = - [\D_\B, b^*]$
for all $b \in \B$ and $[\D_\A, a^\op] \in \End_\B(\E)$ with
$[\D_\A, a^\op]^* = - [\D_\A, (a^\op)^*]$ for all $a \in \A$.
\end{lemma}

\begin{proof}
The right $\A$-linearity of $[\D_\B, b]$ is a straightforward check:
\begin{align*}
[\D_\B, b](ea) - ([\D_\B,b]e)a 
&= \ga_\B \nb_\B(bea) - b\,\ga_\B \nb_\B(ea) - \ga_\B \nb_\B(be)a
+ b(\ga_\B \nb_\B(e)) a
\\
&= \ga_\B \bigl( \nb_\B(bea) - b\,\nb_\B(ea) \bigr)
- \ga_\B \bigl( \nb_\B(be) + b\,\nb_\B(e) \bigr) a
\\
&= \ga_\B(db \ox ea) - \ga_\B(db \ox e) a
\\
&= \ga_\B ((db \ox e)a) - \ga_\B(db \ox e)a = 0,
\end{align*}
since $\ga_\B$ is a right $\A$-module map.

For the adjointability, the previous calculation, with $a = 1$, gives
$$
[\D_\B, b](e) = \ga_\B(db \ox e) = c_\B(db)(e),
$$
and by definition, $c_\B(db) \in \End_\A(\E)$ is assumed adjointable.
Thus also,
$[\D_\B, b]^* = c_\B(db)^* = - c_\B(d(b^*)) = - [\D_\B, b^*]$ since
$c_\B$ is a $*$-homomorphism. The analysis of $[\D_\A, a^\op]$ follows
the same pattern.
\end{proof}

\begin{eg} 
\label{eg:Dirac-bundle}
In the context of Example~\ref{eg:cliff-bundle},
Lemma~\ref{lm:alg-dee} defines a Dirac-type operator $\D$ on the
smooth sections $C^\infty(M,E)$ of the Clifford bundle; that is to
say, a first-order differential operator satisfying
$[\D,f] = c_M(\d f)$ for all $f \in \Coo(M)$. Such Dirac-type
operators typically are of the form $\D = c_M \circ \nb^E$ where
$\nb^E \: \Coo(M,E) \to \Coo(M, T^*M \ox E)$ is a Clifford connection:
see \cite[Sect.~3.3]{BerlineGV}.

We have used the left action of $\Coo(M)$ on $\Coo(M,E)$ to define the
Dirac operator. In this example we could also use the right action to
define another Dirac operator. The relationship between these two
possible definitions plays a prominent role in~\cite{FroehlichGR},
where supersymmetry is used to discuss how one can model additional
geometric structures (oriented, spin, complex, K\"ahler,\dots).
\end{eg}

\begin{rmk} 
Lemma~\ref{lm:alg-dee} gives an algebraic version of the first order
condition for spectral triples (see below) in a very general setting.
We can make some interesting deductions about the operators
$\ga \circ \nb$ that can be defined in this way. The lemma gives the
commutator equations
$$ 
[[\D_\A, a^\op], b] = 0,  \qquad  [[\D_\B, b], a^\op] = 0,
$$
which are respectively equivalent to
$$
[[\D_\A, b], a^\op] = 0,  \qquad  [[\D_\B, a^\op], b] = 0.
$$
This observation allows us to say a little more about commutators with the connection itself as well,
and strengthens the linearity properties of commutators with $\D_\A$ and $\D_\B$.
\end{rmk}

\begin{lemma} 
\label{lm:extra-linear}
For all $a \in \A$, the map
$[\nb_\B, a^\op] : \E \to \Om^1\B \ox_\B \E$ is left $\B$-linear.
Similarly, $[\nb_\A, b] : \E \to \E \ox_\A \Om^1\A$ is right
$\A$-linear for all $b \in \B$.
\end{lemma}

\begin{proof} 
For $a \in \A$, $b \in \B$, $e \in \E$ we find that
\begin{align*}
[\,[\nb_\B, a^\op], b]\,e 
&= [\nb_\B, a^\op]\,be  -b\,[\nb_\B, a^\op]\,e
= \nb_\B(bea) - \nb_\B(be)a - b\,\nb_\B(ea) + b\,\nb_\B(e)a
\\
&= [\nb_\B, b](ea) - ([\nb_\B, b]e)a = [\,[\nb_\B, b], a^\op]\,e = 0,
\end{align*}
since we know from Lemma~\ref{lm:basic-obs} that $[\nb_\B,b]$ is right
$\A$-linear.
\end{proof}

\begin{corl} 
\label{cr:more-commutes}
For any scalars $\la_1,\la_2$, and all $a \in \A$,
$[\la_1 \D_\A + \la_2 \D_\B, a^\op]\in \End_\B(\E)$; and if $b \in B$,
then $[\la_1 \D_\A + \la_2 \D_\B, b]\in \End_\A(\E)$.
\end{corl}

\begin{thm} 
\label{th:dee-is-connection}
Suppose that $\D\: \E \to \E$  satisfies $[\D, a^\op] \in \End_\B(\E)$ and
$[\D, b] \in \End_\A(\E)$, for all $a \in \A$ and $b \in \B$. For
given connections $\nb_\B$ and~$\nb_\A$ on $\E$, there exist module 
maps $\ga_\B$, $\ga_\A$ as in~\eqref{eq:module-maps} and endomorphisms
$T \in \End_\B(\E)$ and $S \in \End_\A(\E)$ such that 
$$
\D = \ga_\A \circ \nb_\A + S = \ga_\B \circ \nb_\B + T.
$$
\end{thm}

\begin{proof}
Given such a $\D$ we define $\ga_\B \: \Om^1\B \ox_\B \E \to \E$ by
$\ga_\B(db \ox e) := [\D, b]\,e$, and likewise define
$\ga_\A \: \E \ox_\A \Om^1\A \to \E$ by 
$\ga_\A(e \ox da) := [\D, a^\op]\,e$.

The associated maps $c_\B \: b_0\,db_1 \mapsto b_0\,[\D, b_1]$ and
$c_\A \: (a_0\,da_1)^\op \mapsto [\D, a_1^\op]\,a_0^\op$ are defined 
on $\Om^1\B$ and $(\Om^1\A)^\op$ respectively. These extend to 
algebra homomorphisms $\Om^\8\B \to \End_\A(\E)$ and
$(\Om^\8\A)^\op \to \End_\B(\E)$, respectively. We may then check that
$$
T := \D - \ga_\B \circ \nb_\B \in \End_\B(\E),  \qquad
S := \D - \ga_\A \circ \nb_\A \in \End_\A(\E).
\eqno \qed
$$
\hideqed
\end{proof}

To deal eventually with general $KK$-classes, we need to take into
account $\Z_2$-graded algebras. So let $\E$ be a $\B$-$\A$-bimodule as
above, but now suppose that $\B$ is a $\Z_2$-graded algebra, and that
$\E$ is $\Z_2$-graded by $\eps$ such that
$\eps b_{\pm} \eps = \pm b_{\pm}$ where $b_+$ and $b_-$ are the even
and odd components of $b \in \B$. We assume that $\A$ commutes with
the grading $\eps$. When we assume that $\E \simeq \B^n p$ is
projective over~$\B$, we shall always take $p = p^2$ in the even
subalgebra of $M_n(\B)$. Denoting the graded commutator by
$[\cdot,\cdot]_\pm$ we obtain the following variant of
Theorem~\ref{th:dee-is-connection}.

\begin{lemma} 
\label{lm:graded-dee}
If $\D\: \E \to \E$  satisfies the graded first-order
condition $[[\D, b]_\pm, a^\op] = 0$, for all $a \in \A$ and
$b \in \B$, then for a given connection $\nb_\B$ on~$\E$, there exists
a module map $\ga_\B$ such that
$\D = \eps\ga_\B \circ \nb_\B  + \eps T$ where $T\: \E \to \E$ is left
$\B$-linear.
\end{lemma}

\begin{proof}
The proof of Theorem \ref{th:dee-is-connection} applies, with only the
following differences. We define $\ga_\B \: \Om^1\B \ox_\B \E \to \E$
by $\ga_\B(db \ox e) := [\eps\D, b]\,e$, using here the ordinary
commutator (to ensure that the Leibniz rule is represented correctly).
Also, $[\eps\D, b_\pm] = \eps[\D,b_\pm]_\mp$. Then, choosing a
connection $\nb_\B$, we find that
$$
T := \eps\,\D - \ga_\B \circ \nb_\B
$$
is $\B$-linear. Hence $\D = \eps\,\ga_\B \circ \nb_\B + \eps\,T$.
\end{proof}

Recall that a connection $\nb$ on a right $\A$-module $\E$ is said to
be \textit{compatible with the $\A$-valued inner product}
$\pairR(\cdot|\cdot)_\A$ if
$$
d\bigl( \pairR(e|f)_\A \bigr) 
= \pairR(e|\nb f)_\A - \pairR(\nb e |f)_\A \word{for all} e,f \in \E,
$$
where on the right hand side the hermitian pairings are extended to
take values in $\Om^1\A$, as follows: if
$\nb f = \sum_i g_i \ox \al_i$ with $g_i \in \E$ and
$\al_i \in \Om^1\A$, then
$$
\pairR(e|\nb f)_\A := \tsum_i \pairR(e|g_i)_\A \,\al_i,  \qquad
\pairR(\nb f|e)_\A := \tsum_i \al_i^*\, \pairR(g_i|e)_\A \,.
$$
For a hermitian left $\B$-module, compatibility of a connection $\nb$ is
expressed, \textit{mutatis mutandis}, by $d\bigl( \pairL_\B(e|f) \bigr)
= - {}\pairL_\B(e|\nb f) + \pairL_\B(\nb e |f)$.

A compatible connection on a finitely generated projective right
$\A$-module $\E = q\A^m$ is of the form
$\nb = q \circ (d \ox 1_m) + A$, where
$A \in \Hom_\A(\E, \E \ox_\A \Om^1\A)$ is self-adjoint, in the sense
that $\pairR(Ae|f)_\A = \pairR(e|Af)_\A$ in $\Om^1\A$, for all
$e,f \in \E$ \cite[Prop.~III.3.6]{Book}.

\subsection{Spectral triples} 
\label{ssc:spec-trips}

In this section we recall the definition of spectral triples and
those basic features and additional properties we need to discuss the
Kasparov product.

\begin{defn} 
\label{df:spec-tri}
A \textit{spectral triple} $(\A,\H,\D)$ consists of a unital%
\footnote{Spectral triples can also be defined over nonunital
algebras; but those are not needed for the present purpose.}
$*$-algebra $\A$, faithfully represented by bounded operators
on a Hilbert space~$\H$ (we write simply~$a$ for the operator 
representing an element $a \in \A$); together with a
selfadjoint operator $\D$ on~$\H$, with dense domain $\Dom\D$, such
that $\Dreg^{-1} \equiv (1 + \D^2)^{-1/2}$ is a compact operator
and, for each $a \in \A$, $a(\Dom\D) \subseteq \Dom\D$ and the 
commutator $[\D, a]$ extends to a bounded operator on~$\H$.

The spectral triple is said to be \textit{even} if there is a
selfadjoint unitary operator $\Ga = \Ga^*$ on~$\H$ (so that
$\Ga^2 = 1$ and thus $\Ga$ determines a $\Z_2$-grading on~$\H$), for
which $[\Ga, a] = 0$ for all $a \in \A$ and $\Ga\D + \D\Ga = 0$. (Note
the consequence for even spectral triples that $[\D, a] \in \A$ only
if $[\D, a] = 0$.) If no such grading is available, the spectral
triple is called~\textit{odd}.
\end{defn}

\begin{rmk} 
Since $\A$ is faithfully represented on~$\H$, we regard $\A$ as a
$*$-subalgebra of $\B(\H)$; its norm closure $A$ is a $C^*$-algebra.
\end{rmk}

\begin{rmk} 
We can talk about \emph{even} spectral triples for $\Z_2$-graded
algebras simply by interpreting all commutators $[\D, a]$, $[\Ga, a]$
as graded commutators. While this is also possible for odd spectral
triples, it is not appropriate from a $KK$-point of view.
\end{rmk}

\begin{eg} 
\label{eg:Dirac-spec-trip}
The Dirac-type operator of a Clifford bundle $E \to M$ on a closed
$C^\infty$ manifold, alluded to in Example~\ref{eg:Dirac-bundle},
gives rise to a spectral triple
$\bigl( \Coo(M), L^2(M,E), \D = c_M \circ \nb^E \bigr)$ over the
algebra $\Coo(M)$.
\end{eg}

\begin{defn} 
\label{df:qc-infty}
The operator $\D$ gives rise to two (commuting) derivations of
operators on~$\H$; we shall denote them by
$$
\d\,T := [\D,T],  \qquad  \dl\,T := [|\D|,T], \word{for} T \in \B(\H).
$$
Note that $\A$ lies within $\Dom\d :=
\set{T \in \B(\H) : T(\Dom\D) \subseteq \Dom\D;\ [\D, T] \in \B(\H)}$.

A spectral triple $(\A,\H,\D)$ is called $QC^\infty$, in the 
terminology of~\cite{CareyPRSHoch}, if
$\A +  \d\A  \subseteq  \Domoo \dl$. (The terms \textit{regular}
\cite{ConnesGrav,Polaris} and \textit{smooth}~\cite{RennieSmooth} are 
synonymous with~$QC^\infty$.)
\end{defn}

\begin{rmk} 
One may replace the derivation $\dl = [|\D|,\cdot]$ by 
$\tilde\dl := [\Dreg,\cdot]$ in the definition of a $QC^\infty$
spectral triple, since $\Domoo \tilde\dl = \Domoo \dl$, as is easily
checked. This is often useful to sidestep issues that arise when
$\ker\D \neq 0$.
\end{rmk}

\begin{defn} 
\label{df:delta-top}
If $(\A,\H,\D)$ is a $QC^\infty$ spectral triple, one can gift $\A$ 
with a locally convex topology, finer than the norm topology of~$A$,
defined by the family of seminorms 
\begin{equation}
q_m(a) := \|\dl^m a\|  \word{and}
q'_m(a) := \|\dl^m([\D,a])\|,  \quad  m = 0,1,2,\dots
\label{eq:dl-snorms} 
\end{equation}
for which the involution $a \mapsto a^*$ is continuous. (Note that
$q_0$ is just the operator norm of~$A$.) Or one can replace the 
seminorms $\set{q_m : m \in \N}$ by the equivalent family of
seminorms~\cite{ConnesRecon}:
$$
p_m(a) := \|\rho_m(a)\|,  \word{where}
\rho_m(a) := \begin{pmatrix}
a & \dl(a) & \cdots & \dl^m(a) \\
0 & a & \ddots & \vdots \\
\vdots & \ddots & a & \dl(a) \\
0 & \cdots & 0 & a  \end{pmatrix}.
$$
The seminorms $p_m$ are submultiplicative: $p_m(ab) \leq
p_m(a)\,p_m(b)$, since $\rho_m$ is a representation of~$\A$. If $\A$
is \textit{complete} in this topology, then $\A$ is a
\textit{Fr\'echet algebra}%
\footnote{A Fr\'echet algebra is defined to be a complete locally
convex algebra whose topology is defined by a countable family of
\textit{submultiplicative} seminorms. Note that the seminorm
$p'_m(a) := p_m([\D,a])$ is not submultiplicative, but the sum
$p_m + p'_m$ will be.}
for which the seminorms $q'_m$ are continuous, by
\cite[Prop.~2.2]{ConnesRecon}.
\end{defn}

Alternatively, if $\A$ is not complete in the topology given by the
seminorms \eqref{eq:dl-snorms}, one can replace $\A$ by its
completion~$\A_\dl$. Assuming that $(\A,\H,\D)$ be a $QC^\infty$
spectral triple, it follows from \cite[Lemma~16]{RennieSmooth} that
$(\A_\dl,\H,\D)$ is also a $QC^\infty$ spectral triple, and moreover
that $\A_\dl$ is a pre-$C^*$-algebra.

Thus, whenever we are given a $QC^\infty$ spectral triple, we may and
shall always assume, by completing its algebra $\A$ if necessary, that
$\A$ is a \textit{Fr\'echet pre-$C^*$-algebra}.

If $T \in \Dom \dl^m$ and $\xi \in \Dom |\D|^m$, then 
$T\xi \in \Dom |\D|^m$ and the equality
\begin{equation}
|\D|^m T \xi = \sum_{k=0}^m \binom{m}{k} \,\dl^k(T)\,|\D|^{m-k}\xi
\label{eq:smooth-switch} 
\end{equation}
holds (by induction on~$m$). There is a similar formula with $\dl$
and~$|\D|$ replaced by $\tilde\dl$ and~$\Dreg$, if desired. Thus, if
$(\A,\H,\D)$ is $QC^\infty$, then the subspace
$$
\H_\infty := \Domoo \D = \Domoo |\D| = \Domoo \Dreg
$$
is mapped to itself by any $a \in \A$.

\begin{defn} 
\label{df:summum-bonum}
Let $(\A,\H,\D)$ be a spectral triple and $\I \subset \K(\H)$ a
(two-sided) symmetric ideal of compact operators. We say that
$(\A,\H,\D)$ is \textit{$\I$-summable} if $\Dreg^{-1} \in \I$.

Let $\L^s = \L^s(\H)$, for $s \geq 1$, be the Schatten ideal of
operators~$T$ for which $|T|^s$ is trace-class. If the spectral triple
is $\L^s$ summable for all $s > p$ (with $p \geq 1$), then
$(\A,\H,\D)$ is \textit{finitely summable}, and the infimum of such
$p$ is called its \textit{spectral dimension}. This holds true in the 
important special cases where we can take $\I = \L^{p,\infty}$ or the 
larger ideal $\Zz_p$ studied in \cite{CareyGRS1} and~\cite{CareyRSS}.
(A positive operator $A$ lies in $\Zz_p$ if and only if 
$A^p \in \Zz_1 = \L^{1,\infty}$, the Dixmier ideal.)
\end{defn}

The next interesting property of spectral triples, namely, the first
order condition, only makes sense for spectral triples defined over
tensor products of algebras.

\begin{defn} 
\label{df:first-order}
The notation $(\A \ox \B, \H, \D)$ for a spectral triple means that
two algebras $\A$ and $\B$ are faithfully represented on~$\H$ by
\textit{commuting} bounded operators, so that the tensor product
$\A \ox \B$ acts on~$\H$ and elements of $\A \ox \B$ have 
bounded commutators with~$\D$.

We say that the spectral triple $(\A \ox \B, \H, \D)$ satisfies the
\textit{first order condition} if $[[\D,a], b] = 0$ for all $a \in \A$
and $b \in \B$.
\end{defn}

\begin{defn} 
\label{df:best-algebra}
If $(\A,\H,\D)$ is an even spectral triple, we shall use the notation
$\Cc \equiv \CDA$ for the $\Z_2$-graded subalgebra of $\B(\H)$
generated by $\A$ (of even degree) and $\set{[\D,a] : a \in \A}$ (of
odd degree). There is an algebra homomorphism
$\pi_\D \: \Om^\8\A \to \CDA$ given by
\begin{equation}
\pi_\D(a_0\,da_1 \cdots da_k) := a_0\,[\D, a_1]\cdots [\D, a_k].
\label{eq:best-homom} 
\end{equation}
If $(\A,\H,\D)$ is an odd spectral triple, we can consider the even
spectral triple $(\A, \H\oplus\H, \D')$ where $\A$ acts diagonally,
$\D' = \twobytwo{\D}{0}{0}{-\D}$, and the grading is given by
$\twobytwo{0}{1}{1}{0}$. Then we obtain a $\Z_2$-grading on
$\mathcal{C}_{\D'}(\A)$.

Recall that $\Om^\8\A$ becomes an involutive algebra by setting
$(da)^* := - d(a^*)$; then $\pi_\D$ is a $*$-rep\-resentation of the
differential forms, and so also the Hochschild chains~\cite{Loday}, of
$\A$ by operators on~$\H$.
\end{defn}

\begin{prop} 
\label{pr:first-order}
Let $(\A,\H,\D,\Ga)$ be an even spectral triple and suppose that
$\H_\infty$ is a finite projective left $\A$-module and that
$\H = L^2(\H_\infty,\psi)$. Then with
\begin{equation}
\B = \set{T \in \B(\H) :  T(\H_\infty) \subseteq \H_\infty, \
[T,\Ga] = 0, \ [T,w] = 0 \text{ for } w \in \CDA},
\label{eq:first-order} 
\end{equation}
we find that $(\A \ox \B,\H,\D,\Ga)$ is an even spectral triple
satisfying the first order condition.
\end{prop}

\begin{proof}
All we need to show is the boundedness of $[\D, T]$ for $T \in \B$. 
However for $a \in \A$ we get
$$
[[\D, T], a] = - [T, [\D, a]] = 0,
$$
because $T$ commutes with $\CDA$ by assumption. Hence $[\D,T]$ is left
$\A$-linear and maps $\H_\infty$ to itself, so
Proposition~\ref{pr:lin-bdd} implies that $[\D, T]$ is bounded. The
first order condition is obvious.
\end{proof}

\begin{eg} 
\label{eg:Dirac-first-order}
In the context of Examples \ref{eg:Dirac-bundle}
and~\ref{eg:Dirac-spec-trip}, Proposition~\ref{pr:first-order} shows
that the data 
$\bigl( \Coo(M) \ox \Coo(M), L^2(M,E), \D = c_M \circ \nb^E \bigr)$
form a spectral triple. This follows since both the left and right
actions of $\Coo(M)$ commute with the action of the Clifford algebra,
and moreover the smooth sections of $E$ are finite projective over
$\Coo(M)$ and form the smooth domain of $\D$; see~\cite{RennieSmooth}.
\end{eg}

\begin{rmk} 
If $(\A,\H,\D,\Ga)$ is $QC^\infty$ then $(\A \ox \B,\H,\D,\Ga)$ is
$QC^\infty$ for the action of $\A$, but not necessarily for  the 
action of~$\B$. Thus we shall say ``$QC^\infty$ for~$\A$'' in such 
cases, when considering spectral triples defined over a tensor 
product of algebras.
\end{rmk}

Later we will also want some information about the $\A$-module
structure of $\CDA$. It is immediate that $\CDA$ is an $\A$-bimodule.
Regarding $\A$-valued inner products on $\CDA$, the following result
is helpful. Recall that an \textit{operator-valued weight} is a
positive linear map $\Psi\: \Cc \to \A$ from a $*$-algebra $\Cc$ onto
a $*$-subalgebra $\A$ that satisfies $\Psi(awb) = a\,\Psi(w)\,b$ for
$w \in \Cc$ and $a,b \in \A$; thus it behaves like a conditional
expectation except that it need not be unit-preserving
\cite[Appendix~A]{Kosaki}, nor need it extend to the $C^*$-completion
of~$\Cc$ as a bounded map.

\begin{lemma} 
\label{lm:op-weight}
Let $\Cc$ be a unital $*$-algebra and let $\A$ be a unital
$*$-subalgebra with the same unit~$1$ (i.e., the inclusion
$\A \hookto \Cc$ is unit-preserving). The existence of a left
$\A$-valued inner product $\pairL_\A(\cdot|\cdot)$ on $\Cc$ such that
right multiplication of $\Cc$ on itself defines an adjointable action
is equivalent to the existence of a faithful operator-valued weight
$\Psi\: \Cc \to \A$.
\end{lemma}

\begin{proof}
Suppose first that $\pairL_\A(\cdot|\cdot)$ is an $\A$-valued left
inner product on $\Cc$. Define
$$
\Psi:\Cc \to \A \word{by} \Psi(w) := \pairL_\A(w1|1) = \pairL_\A(w|1).
$$
Then, assuming right multiplication to be adjointable, we get, for
$a,b \in \A$ and $w \in \Cc$,
\begin{align*}
\Psi(a w b) 
&= \pairL_\A(awb|1) = a\,\pairL_\A(wb|1) 
\\
&= a\,\pairL_\A(w|b^*) = a\,\pairL_\A(w|1)\,b = a\,\Psi(w)\,b.
\end{align*}
Similarly, positivity of~$\Psi$ follows from
$$
\Psi(w^*w) = \pairL_\A(w^*w|1) = \pairL_\A(w^*|w^*) \geq 0,
$$
with equality if and only if $w^* = 0$, if and only if $w = 0$.

Conversely, suppose that a faithful operator valued weight
$\Psi : \Cc \to \A$ is given. Define
$$
\pairL_\A(u|v) := \Psi(u v^*),   \word{for}  u,v \in \Cc.
$$
One verifies easily that this defines a positive definite left
$\A$-module hermitian form on~$\Cc$, for the left action of~$\A$
coming from the inclusion $\A \subset \Cc$. The adjointability of
right multiplication by~$\Cc$ is clear.
\end{proof}

\begin{rmk} 
If $\pairL_\A(1|1) = 1 \in \A$ then the corresponding
operator-valued weight~$\Psi$ is an expectation, and conversely.
\end{rmk}

\subsection{Kasparov products using bimodule connections} 
\label{ssc:kas-prods}

In order to make the passage from spin$^c$ to Riemannian manifolds, we
will employ unbounded Kasparov products as described
in~\cite{Kucerovsky}.

In this subsection we assume that $(\A \ox \B^\op, \H, \D, \eps)$ is
an even spectral triple (in the $\Z_2$-graded sense for the
algebra~$\B$) satisfying smoothness and first order conditions. We
also require that $\H_\infty$ be finitely generated and projective as
a right module over~$\B$.

{}From subsection~\ref{ssc:Dirac-type}, this means that we can
represent $\D$ using a $\B$-compatible connection, so that
$\D = \eps(\ga \circ \nb_\B^\H)+\eps\,T$ where $T$ is $\B$-linear,
$\nb_\B^\H : \H_\infty \to \H_\infty \ox_\B \Om^1\B$, and
\begin{equation}
\ga : \H_\infty \ox_\B \Om^1\B \to \H_\infty  \word{is given by}
\ga(\xi \ox b_0\,db_1) = [\eps\D, b_1^\op](\xi\,b_0).
\label{eq:right-thinking} 
\end{equation}
(The grading $\eps$ is required only when $\B$ is $\Z_2$-graded and we
employ a graded first order condition. Otherwise, just put
$\eps = 1$.) Note that $\ga$ is right $\B$-linear since
\begin{align*}
\ga(\xi \ox b_0\,db_1\,b_2)
&= \ga(\xi \ox b_0\,d(b_1b_2)) - \ga(\xi \ox b_0b_1\,db_2)
\\
&= [\eps\D, b_2^\op b_1^\op](\xi\,b_0) - [\eps\D, b_2^\op](\xi\,b_0b_1)
= b_2^\op \,[\eps\D, b_1^\op](\xi\,b_0)
\\
&= [\eps\D, b_1^\op](\xi\,b_0)b_2 = \ga(\xi \ox b_0\,db_1)b_2.
\end{align*}

\textit{Assume now that we are also given a $\B$-$\Cc$-bimodule $\E$}
which is finitely generated and projective as a left $\B$-module, so
$\E \isom \B^n q$ where $q\in M_n(\B)$ is a projector. We take the
$\B$-valued inner product $\pairL_\B(\cdot|\cdot)$ on $\E$ given by
this identification. Choose an (arbitrary but fixed) connection
$\nb_\B^\E \: \E \to \Om^1\B \ox_\B \E$ compatible with this inner
product. We let $\eps'$ be a $\Z_2$-grading of $\E$ such that
$$
\eps' b_\pm \eps' = \pm\,b_\pm \text{ for } b \in \B,  \qquad
c \eps' = \eps' c \text{ for } c\in \Cc. 
$$
We marry the connections $\nb_\B^\H$ and $\nb_\B^\E$ in the usual way,
by defining a linear map $\nb_\B$ on the balanced tensor product
$\H_\infty \ox_\B \E$ by
$$
\nb_\B : \H_\infty \ox_\B \E \to \H_\infty \ox_\B \Om^1\B \ox_\B \E,
\qquad
\nb_\B(\xi \ox e) := \nb_\B^\H(\xi) \ox e + \xi \ox \nb_\B^\E(e).
$$
To see that $\nb_\B$ is indeed well defined, we remark that
\begin{align*}
\nb_\B(\xi b \ox e)
&=  \nb_\B^\H(\xi b) \ox e + \xi b \ox \nb_\B^\E(e)
= \nb_\B^\H(\xi) b \ox e + \xi \ox db \ox e + \xi \ox b\,\nb_\B^\E(e)
\\
&= \nb_\B^\H(\xi) \ox be + \xi \ox \nb_\B^\E(be)
= \nb_\B(\xi \ox be).
\end{align*}
Similarly the operator $\eps\ox\eps'$ is well-defined on the tensor
product.

In consequence, the linear operator
$\Dhat \: \H_\infty \ox_\B \E \to \H_\infty \ox_\B \E$ given by
$$
\Dhat := (\eps\ox\eps')(\ga \ox 1_\E) \circ \nb_\B+\eps\,T\ox\eps'
$$
is also well defined. 

\vspace{6pt}

To examine $\Dhat$ more closely, it helps to work in a framework where
the isomorphism $\E \isom \B^n q$ is explicit. If 
$e = (b_1,\dots,b_n) = (b_1,\dots,b_n)q \in \E$, we write
$$
\phi : \H_\infty \ox_\B \E \to (\H_\infty \ox \C^n)q, \qquad
\phi(\xi \ox e) := (\xi b_1,\dots,\xi b_n)q = (\xi b_1,\dots,\xi b_n),
$$
and likewise
\begin{gather*}
\phihat : \H_\infty \ox_\B \Om^1\B \ox_\B \E 
\to \H_\infty \ox_\B (\Om^1\B)^n q,
\\
\phihat(\xi \ox \om \ox e) := (\xi \ox \om b_1,\dots,\xi \ox \om b_n)q
= (\xi \ox \om b_1,\dots,\xi \ox \om b_n).
\end{gather*}

\begin{lemma} 
\label{lm:big-dee}
The linear map $\phi$ is an $\A$-$\Cc$-bimodule isomorphism satisfying
\begin{enumerate}
\item[\textup{(a)}]
$\phi \circ (\ga \ox 1_\E) = (\ga \ox 1_n) \circ \phihat\,$; and
\item[\textup{(b)}]
$\phi \circ \Dhat \circ \phi^{-1}
= q^\op(\D \ox \eps' 1_n)q^\op + \Ahat$, where $\Ahat$ is a bounded and
selfadjoint operator on $(\H \ox \C^n)q = \H^n q$.
\end{enumerate}
Moreover, $\Dhat$ is a selfadjoint operator on the Hilbert space
$\H \ox_\B \E$.
\end{lemma}

\begin{rmk} 
Since $1\in\B$ and $q\in M_n(\B)$ are even elements, we can define a
$\Z_2$-grading on $\B^nq$ by setting
$(b_1,\dots,b_n)_+ := (b_{1+},\dots,b_{n+})$. If, abusing notation, we
denote this grading by $\eps'$, then 
$\phi \circ (\eps \ox \eps') = (\eps \ox \eps') \circ \phi$. 
Similarly, if $T \: \H_\infty \to \H_\infty$ is $\B$-linear,
$\phi \circ (\eps\,T \ox \eps') = (\eps\,T \ox \eps') \circ \phi$.
\end{rmk}

\begin{proof}
The left $\A$-linearity of~$\phi$ is obvious; its right 
$\Cc$-linearity comes from that of the identification 
$\E \isom \B^n q$.

Ad~(a): Using the right $\B$-linearity of~$\ga$, it suffices to
evaluate
\begin{align*} 
\phi \circ (\ga \ox 1_\E)(\xi \ox \om \ox e) 
&= \phi(\ga(\xi \ox \om) \ox e)
= (\ga(\xi \ox \om)b_1, \dots, \ga(\xi \ox \om)b_n)
\\
&= (\ga(\xi \ox \om b_1), \dots, \ga(\xi \ox \om b_n))
= (\ga \ox 1_n)(\xi \ox \om b_1, \dots, \xi \ox \om b_n)
\\
&= (\ga \ox 1_n) \circ \phihat(\xi \ox \om \ox e).
\end{align*}

Ad~(b): We  write $\nb_\B^\E = q^\op(d \ox 1_n) + A$ with
$A \in \Hom_\B(\E, \Om^1\B \ox_\B \E)$, and let $u_k$ denote the
standard unit vectors in $\C^n$. Then we find that
\begin{align*}
q^\op(\ga \ox 1_n) &(\nb_\B^\H \ox 1_n)q^\op \bigl( 
\phi(\xi \ox e) \bigr)
= q^\op(\ga \ox 1_n)\bigl( 
(\nb_\B^\H \ox 1_n)(\xi b_1,\dots,\xi b_n) \bigr)
\\
&= q^\op(\ga \ox 1_n)\bigl(
\nb_\B^\H(\xi)b_1, \dots, \nb_\B^\H(\xi)b_n \bigr)
+ q^\op(\ga \ox 1_n)(\xi \ox db_1, \dots, \xi \ox db_n)
\\
&= \bigl( \ga(\nb_\B^\H(\xi))b_1,\dots,\ga(\nb_\B^\H(\xi))b_n \bigr) q
+ \bigl( \ga(\xi \ox db_1), \dots, \ga(\xi \ox db_n) \bigr) q
\\
&= (\ga \ox 1_n)\bigl( \phihat(\nb_\B^\H(\xi) \ox e) \bigr)
+ \tsum_{j,k=1}^n (\ga \ox 1_n) \bigl( 
\phihat(\xi \ox db_j \ox u_k q_{kj}) \bigr)
\\
&= \phi(\ga \ox 1_\E) \bigl( \nb_\B^\H(\xi) \ox e
+ \tsum_{j,k=1}^n \xi \ox db_j \ox u_k q_{kj} \bigr)
\\
&= \phi(\ga \ox 1_\E) \bigl( \nb_\B^\H(\xi) \ox e
+ \xi \ox \nb_\B^\E(e) - \xi \ox A(e) \bigr)
\\
&=\phi(\ga \ox 1_\E) \bigl( \nb_\B(\xi \ox e) - \xi \ox A(e) \bigr).
\end{align*}
Consequently, multiplying by $\eps \ox \eps'$ we find 
$$
q^\op (\D\ox \eps' 1_n) q^\op (\phi(\xi\ox e))
= \phi(\Dhat(\xi \ox e)) - \phi\bigl( (\eps\ox\eps') (\ga \ox 1_\E)
(1_{\H_\infty} \ox A) (\xi \ox e) \bigr).
$$
On setting $\Ahat := \phi \circ (\eps \ox \eps') \circ (\ga \ox 1_\E)
\circ (1_{\H_\infty} \ox A) \circ \phi^{-1}$, we find that
$$
\phi \circ \Dhat \circ\phi^{-1} = q^\op(\D \ox \eps' 1_n)q^\op + \Ahat
$$
as operators on $(\H_\infty \ox \C^n)q = \H_\infty^n q$.

We may now write $A = \sum_{i,j=1}^n \om_{ij} \ox e_{ij}$ where the
$e_{ij}$ are matrix units and $\om_{ij} \in \Om^1\B$. Mindful
of~\eqref{eq:right-thinking}, for $\rho = b_0\,db_1 \in \Om^1\B$ we
write $c(\rho^\op) := [\eps\D, b_1^\op]\,b_0^\op$ for the right
$\B$-linear operator on~$\H_\infty$ corresponding to~$\rho$. Using
Theorem~\ref{th:dee-is-connection}, the selfadjointness of~$\D$ and
the $\B$-compatibility of the connection $\nb_\B^\H$ imply that
$c(\om_{ij}^\op)^* = c(\om_{ji}^\op)$ for each~$i,j$. Then
$$
\Ahat\biggl( \sum_{j=1}^n \xi b_j \,u_j \biggr) 
= \sum_{i,j=1}^n c(\om_{ij}^\op)(\xi) b_i\,u_j.
$$
The scalar product on $\H_\infty^n q$ is given by
\begin{align}
\biggl< \sum_j \xi b_j\,u_j \biggm| \sum_k \eta  b'_k\,u_k \biggr> 
= \sum_j \scal<\xi b_j|\eta b_j'>
= \sum_{j,k} \scal<\xi b_j q_{jk}|\eta b'_k> \,.
\label{eq:scal-prod} 
\end{align}
With these formulae, it is straightforward to check that $\Ahat$ is a
symmetric operator on the Hilbert space $\H^n q$ (the completion of
$\H_\infty^n q$ for this scalar product). Now, ignoring
$\eps \ox \eps'$, $\Ahat$ is just multiplication by the matrix with
bounded entries $c(\om_{ij}^\op)$, which is manifestly bounded; so
$\Ahat$ is a bounded selfadjoint operator on~$\H^n q$.

The Hilbert space $\H \ox_\B \E$ is (by definition) the completion of 
$\H_\infty \ox_\B \E$ in the corresponding scalar product, so that 
$\phi$ extends to a unitary isomorphism from $\H \ox_\B \E$ to
$\H^n q$. To show that $\Dhat$ is a selfadjoint operator on 
$\H \ox_\B \E$, it is thus enough to show that 
$q^\op(\D \ox \eps' 1_n)q^\op$ is a selfadjoint operator on $\H^n q$,
since $\Ahat$ delivers a bounded selfadjoint perturbation of it.

Write $\D_n \equiv \D \ox \eps' 1_n$. Observe that $q^\op \D_n q^\op$
is symmetric on the dense domain
$$
\Dom(q^\op \D_n q^\op) 
= \set{\xi \in \Dom \D_n \subset \H^n : \xi q = \xi}.
$$
The domain of the adjoint $\Dom((q^\op \D_n q^\op)^*)$ on $\H^n q$
consists of all $\xi \in \H^n q$ such that for all
$\eta \in \Dom(q^\op \D_n q^\op)$ there is some $\zeta \in \H^n q$ for
which $\scal< q^\op \D_n q^\op \eta | \xi> = \scal<\eta | \zeta>$.
However, we see that
$$
\scal< q^\op \D_n q^\op \eta | \xi>
= \scal< \D_n q^\op \eta | q^\op \xi> = \scal<\D_n\eta | \xi> \,.
$$
Therefore
$$
\Dom((q^\op \D_n q^\op)^*) 
= \set{\xi \in \Dom \D_n^* \subset \H^n : \xi q = \xi}
= \Dom(q^\op \D_n q^\op) 
$$
since $\D_n$ is selfadjoint on~$\H^n$. Thus $q^\op \D_n q^\op$ is
selfadjoint.
\end{proof}

We now come to an important point: the passage from $\D$ to~$\Dhat$ 
does not change the order of summability of the corresponding 
spectral triples.

\begin{prop} 
\label{pr:summa}
If $(\A \ox \B^\op, \H, \D)$ is $\I$-summable, then so also is
$(\A, \H \ox_\B \E, \Dhat)$, for either of the ideals $\I = \L^s$,
$s \geq 1$; or $\I = \Zz_p$, $p \geq 1$.
\end{prop}

\begin{proof}
First consider the summability of the spectral triple
$(\A, \H^n q, q^\op \D_n q^\op)$. Note that
\begin{align*}
(q^\op \D_n q^\op)^2
&= q^\op \D_n q^\op \D_n q^\op
= q^\op [\D_n, q^\op] \D_n q^\op + q^\op \D_n^2 q^\op
\\
&= q^\op [\D_n, q^\op][\D_n, q^\op] + q^\op [\D_n, q^\op] q^\op \D_n
+ q^\op \D_n^2 q^\op
= q^\op [\D_n, q^\op]\,[\D_n, q^\op] + q^\op \D_n^2 q^\op,
\end{align*}
because $p\,\dl(p)p = 0$ for any projector~$p$ and any derivation
$\dl$ with $p \in \Dom\dl$.

Since $q^\op$ acts as the identity operator on $\H^n q$, we find that
\begin{equation}
\reg{q^\op \D_n q^\op}^{-1} = (q^\op + (q^\op \D_n q^\op)^2)^{-1/2}
= q^\op \bigl( 1 + \D_n^2 + q^\op[\D_n, q^\op][\D_n, q^\op]
\bigr)^{-1/2} q^\op.
\label{eq:squash-reg-inverse} 
\end{equation}

So suppose that $\Dreg^{-1} = (1 + \D^2)^{-1/2}$ is contained in the
symmetric ideal $\I = \I(\H)$ where, say, $\I = \L^s$ or~$\Zz_p$. Then
$\reg{\D_n}^{-1} = (1 + \D_n^2)^{-1/2}$ lies in $\I(\H^n)$, and so 
$q^\op (1 + \D_n^2)^{-1/2} q^\op$ lies in $\I(\H^n q)$.

{}From \eqref{eq:squash-reg-inverse} it follows that
\begin{align*}
\reg{q^\op \D_n q^\op}^{-2} 
&= (q^\op + (q^\op \D_n q^\op)^2)^{-1}
= \bigl( q^\op(1 + \D_n^2)q^\op
+ q^\op[\D_n, q^\op][\D_n, q^\op] \bigr)^{-1}
\\
&= q^\op (1 + \D_n^2)^{-1} q^\op
- q^\op (1 + \D_n^2)^{-1} q^\op[\D_n, q^\op]\,[\D_n, q^\op]\, \bigl(
q^\op + (q^\op \D_n q^\op)^2 \bigr)^{-1}.
\end{align*}
This shows that if $q^\op(1 + \D_n^2)^{-1} q^\op \in \I^2$ then
$(q^\op + (q^\op \D_n q^\op)^2)^{-1} \in \I^2$, too. For 
$\I = \L^s$ or $\I = \Zz_p$, this then implies that
$\reg{q^\op \D_n q^\op}^{-1} \in \I$, as desired.

To finish, in view of Lemma~\ref{lm:big-dee}, we only need to show 
that if $A$ is a bounded selfadjoint operator on~$\H$ and 
$(1 + \D^2)^{-1/2} \in \I(\H)$, then
$(1 + (\D + A)^2)^{-1/2} \in \I(\H)$, too. Observe that 
$$
(i + \D)^{-1} = (1 + \D^2)^{-1/2} (1 + \D^2)^{1/2} (i + \D)^{-1},
$$
and since $(1 + \D^2)^{1/2} (i + \D)^{-1}$ is unitary, we can start 
from $(i + \D)^{-1} \in \I(\H)$. Using the identity
\begin{equation}
(i + \D + A)^{-1} = (i + \D)^{-1} - (i + \D + A)^{-1} A(i + \D)^{-1},
\label{eq:inverse-perturbed} 
\end{equation}
we conclude that $(i + \D + A)^{-1} \in \I(\H)$ and thus
$(1 + (\D + A)^2)^{-1/2} \in \I(\H)$.
\end{proof}

Having now constructed $\Dhat$, one could expect that since $\A$ and 
$\Cc$ have commuting actions on $\H \ox_\B \E$, one could produce a 
spectral triple over the algebra $\A \ox \Cc^\op$. In order to obtain 
it, we need finite projectivity under the left action of~$\A$.

\begin{thm} 
\label{th:nice-product}
Let the even spectral triple $(\A \ox \B^\op, \H, \D)$ be $QC^\infty$
for~$\A$, $\Z_2$-graded for the $\Z_2$-graded algebra $\B^\op$, and
$\I$-summable; and let $\E$ be a $\B$-$\Cc$-bimodule, finitely
generated and projective over~$\B$, and $\Z_2$-graded for~$\B$. Then
the associated spectral triple $(\A, \H \ox_\B \E, \Dhat)$ is
$QC^\infty$ and $\I$-summable, and has the same spectral dimension.

Moreover, if $\H_\infty = \Domoo \D$ is finitely generated and
projective as a left $\A$-module and if $\H = L^2(\H_\infty,\psi)$
for some positive linear functional $\psi$ on~$\A$, then
$(\A \ox \Cc^\op, \H \ox_\B\E, \Dhat)$ is a spectral triple satisfying
the first order condition, and it is $QC^\infty$ for the action
of~$\A$.
\end{thm}

\begin{proof}
The left actions of $\A$ on $\H$ and $\H \ox_\B \E$ satisfy
\begin{equation}
[\Dhat, a] = (\eps\ox\eps')\,(\ga \ox 1_\E)\,[\nb_\B, a]
= \eps\ga\,[\nb_\B^\H, a] \ox \eps' = [\D, a] \ox \eps',
\label{eq:si-beag-si-mor} 
\end{equation}
since $\ga$ is left $\A$-linear: recall that we have assumed that the
spectral triple $(\A \ox \B^\op, \H, \D)$ satisfies the (graded) first
order condition. Thus the commutators $[\Dhat, a]$ are bounded and
generate a representation of $\CDA$ on $\H \ox_\B \E$. The right
action of $\Cc$ on~$\E$ extends in the obvious way to a representation
of $\Cc^\op$ on $\H \ox_\B \E$ commuting with this left action
of~$\CDA$.

We can identify $\H \ox_\B \E$ with $\H^n q$ as before ---omitting the
explicit isomorphism $\phi$ from the notation--- so that $\Dhat^2$ is
a bounded perturbation of $q^\op(\D^2 \ox 1_n)q^\op$. Since $q^\op$
commutes with the left action of $\A$, it follows that $a$ and 
$[\Dhat, a]$ lie in $\Domoo\bigl[ |\Dhat|, \cdot \bigr]$ so that 
$(\A, \H \ox_\B \E, \Dhat)$ is indeed $QC^\infty$.

For example, a short calculation shows that
$$
\bigl (q^\op + q^\op (\D^2 \ox 1_n) q^\op \bigr)^{-1/2}
\bigl[ q^\op (\D^2 \ox 1_n) q^\op, [\D,a] \ox \eps' \bigr]
= q^\op\bigl( (1 + \D^2)^{-1/2} [\D^2,[\D,a]] \ox \eps' \bigr) q^\op.
$$
The right hand side is bounded by the $QC^\infty$ property of
$(\A,\H,\D)$; and the consequent boundedness of the left hand side 
yields the $QC^\infty$ property for the other spectral triple: see, 
for instance, \cite[Lemma~13.2]{ConnesRecon}.

The $\I$-summability of the associated spectral triple has already 
been established by Proposition~\ref{pr:summa}.

If $\H_\infty$ is finitely generated and projective as a left
$\A$-module, then so too is $\H_\infty \ox_\B \E \isom \H_\infty^n q$,
which is a direct summand of $\H_\infty^n$. The boundedness of
$[\Dhat, c^\op]$, for $c \in \Cc$, follows since the relation
$$
[[\Dhat, c^\op], a] = [[\Dhat, a], c^\op] = 0, \word{for all} a \in \A
$$
shows that $[\Dhat, c^\op]$, which maps $\H_\infty \ox_\B \E$ to 
itself, is left $\A$-linear. It is easy to check that 
$\H = L^2(\H_\infty,\psi)$ implies that
$\H^n q = L^2(\H_\infty^n q,\psi)$; thus Proposition~\ref{pr:lin-bdd}
is applicable,%
\footnote{Or rather, we need the obvious right-to-left variant of that
Proposition.}
and establishes the boundedness of $[\Dhat, c^\op]$.
\end{proof}

\begin{eg} 
\label{eg:twisted-bundle}
Let $E \to M$ be a Clifford bundle as in Example~\ref{eg:cliff-bundle}
and, given a Clifford connection $\nb^E$ on~$E$, let
$\bigl( \Coo(M) \ox \Coo(M), L^2(M,E), \D = c_M \circ \nb^E \bigr)$ be
the associated spectral triple (see Examples \ref{eg:Dirac-bundle},
\ref{eg:Dirac-spec-trip}, \ref{eg:Dirac-first-order} and
Lemma~\ref{lm:alg-dee}). Given another vector bundle $F \to M$, we may
form a connection $\nb^{E\ox F} := \nb^E \ox 1_F + 1_E \ox \nb^F$ on
$E \ox F$ and a Clifford action $c_M \ox 1_F$ on $E \ox F$. This
yields a new ``twisted'' Clifford bundle and so a spectral triple
$\bigl( \Coo(M), L^2(M, E \ox F),
(c_M \ox 1_F) \circ \nb^{E\ox F} \bigr)$. 
\end{eg}

In \cite{Kucerovsky}, Kucerovsky gives sufficient conditions for an
unbounded Kasparov module to represent the Kasparov product of two
other unbounded Kasparov modules. A spectral triple is tantamount to
an unbounded Kasparov module; a finitely generated projective
(bi-)module may also be regarded as another, where the zero operator
takes the place of the operator~$\D$. We use this theory in the next
proposition to show that we are computing Kasparov products when we
twist our spectral triple by such bimodules. The full force of
Kucerovsky's conditions is not needed in this setting, since the
product simplifies considerably if one of the operators is zero.

\begin{prop} 
\label{pr:married-triple}
Assume now that $(\A \ox \B^\op, \H_\infty, \D)$ is an even spectral
triple which is $QC^\infty$ for~$\A$, and that $\H_\infty$ is finitely
generated and projective over both $\A$ and~$\B$. Let $\E$ be a
$\B$-$\Cc$-bimodule, finitely generated and projective over~$\B$, and
$\Z_2$-graded for~$\B$. Then the spectral triple
$(\A \ox \Cc^\op, \H \ox_\B \E, \Dhat)$ represents the Kasparov
product
$[\E] \ox_{B^\op} [(\A \ox \B^\op, \H, \D)] \in KK(A \ox C^\op, \C)$
of the class $[\E] \in KK(C^\op, B^\op)$ and the class
$[(\A \ox \B^\op, \H, \D)] \in KK(A \ox B^\op, \C)$.
\end{prop}

\begin{proof}
In order to handle the product of modules, it should be noted that the
inner product making a hermitian left $\B$-module $\E$ into a
hermitian right $\B^\op$-module $\E'$ is given by
$$
\pairR(e|f)_{\B^\op} := \bigl( \pairL_\B(e|f)^* \bigr)^\op.
$$
The module underlying the Kasparov product is $\E' \ox_{\B^\op} \H$,
but this is cumbersome and hardly intuitive. Unpacking the definitions
of the scalar product for $\E' \ox_{\B^\op} \H$ yields the formula of
\eqref{eq:scal-prod}. This means that there is an isomorphism of
Hilbert spaces from $\E' \ox_{\B^\op} \H$ to $\H \ox_\B \E$ given by
$e \ox_{\B^\op} \xi \mapsto \xi \ox_\B e$, for $e \in \E$,
$\xi \in \H$. It is easy to see that this isomorphism intertwines the
actions of $\A$ and $\Cc$ on these Hilbert spaces. As a result, the
module underlying the Kasparov product is $\H \ox_\B \E \isom \H^n q$.

Having dealt with the underlying module, we are left with checking the
three conditions of \cite[Thm.~13]{Kucerovsky}. Of these three
conditions, the domain compatibility and positivity constraints of
that theorem are both trivial, since the operator making the left
$\B$-module $\E$ a Kasparov module is zero. Thus we are left with
checking the connection condition.

This condition requires that for a dense set of $e$ in (the
$C^*$-completion of) $\E$, with $e$ homogeneous of degree $\del e$
(i.e., even or odd), the graded commutator
$$
\Biggl[ \twobytwo{\Dhat}{0}{0}{\D},
\twobytwo{0}{T_e}{T_e^*}{0} \Biggr]_\pm 
:= \twobytwo{\Dhat}{0}{0}{\D}\, \twobytwo{0}{T_e}{T_e^*}{0} -
(-1)^{\del e} \twobytwo{0}{T_e}{T_e^*}{0}\, \twobytwo{\Dhat}{0}{0}{\D}
$$
should be bounded on $\Dom \Dhat \oplus \Dom \D$, where 
$T_e\: \H_\infty \to \H_\infty \ox_\B \E$ is given by
$T_e(\rho) := \rho \ox e$, with adjoint 
$T_e^*(\xi \ox f) := \xi\,\pairL_\B(f|e)$. 

With $u_1,\dots,u_n$ the standard basis vectors in $\B^n$; and
$\xi \ox f \in \H_\infty \ox_\B \E$; and
$e = \sum_i e_i\,u_i \in \B^n q$ even or odd; and with
$\eta \in \H_\infty$, we compute
\begin{align*}
&\Biggl[ \twobytwo{\Dhat}{0}{0}{\D},
\twobytwo{0}{T_e}{T_e^*}{0} \Biggr]_\pm 
\begin{pmatrix} \xi \ox f \\ \eta \end{pmatrix}
= \twobytwo{\Dhat}{0}{0}{\D}
\begin{pmatrix} \eta \ox e \\ \xi\,\pairL_\B(f|e) \end{pmatrix}
-(-1)^{\del e} \twobytwo{0}{T_e}{T_e^*}{0}
\begin{pmatrix} \Dhat(\xi \ox f) \\ \D \eta \end{pmatrix}
\\[\jot]
&\qquad\qquad= 
\begin{pmatrix} \Dhat(\eta \ox e) - (-1)^{\del e} \D\eta \ox e \\[\jot]
\D\bigl( \xi\,\pairL_\B(f|e) \bigr)
- (-1)^{\del e} T_e^*(\Dhat(\xi \ox f)) \end{pmatrix}
\\[\jot]
&\qquad\qquad=
\begin{pmatrix} \sum_i \D(\eta e_i) \ox u_i
+ (\eps\,T \ox \eps' + \Ahat)(\eta \ox e) - (-1)^{\del e} \D\eta \ox e
\\[\jot]
\D\bigl( \xi\,\pairL_\B(f|e) \bigr)
- (-1)^{\del e} \sum_i \D(\xi f_i) \pairL_\B(u_i|e)
- (-1)^{\del e} T_e^* (\eps\,T \ox \eps' + \Ahat)(\xi \ox f)
\end{pmatrix}
\\[\jot]
&\qquad\qquad=
\begin{pmatrix} \sum_i[\D, e_i^\op]_\pm \,\eta \ox u_i
+ (\eps\,T \ox \eps' + \Ahat) T_e(\eta) \\[\jot]
\sum_i [\D, e_i^\op]_\pm \,\xi f_i
- (-1)^{\del e} T_e^* (\eps\,T \ox \eps' + \Ahat)(\xi \ox f)
\end{pmatrix}
= \begin{pmatrix} R(\eta) \\[\jot] S(\xi \ox f) \end{pmatrix},
\end{align*}
where $R$ and $S$ are bounded operators. Thus the graded commutator is
bounded for each $e \in \B^n q$. Thus the connection condition is
satisfied and therefore $(\A \ox \Cc^\op, \H \ox_\B \E, \Dhat)$
represents the Kasparov product.
\end{proof}

\begin{rmk} 
\label{rk:twisted-bundle}
In the context of Example~\ref{eg:twisted-bundle}, this Proposition
asserts that the $KK$ class of the spectral triple
$\bigl( \Coo(M), L^2(M, E \ox F),
(c_M \ox 1_F) \circ \nb^{E\ox F} \bigr)$ is precisely the Kasparov
product of the classes of $F$ and
$\bigl( \Coo(M) \ox \Coo(M), L^2(M,E), c_M \circ \nb^E \bigr)$.
\end{rmk}

Others have observed the utility of the unbounded version of the
Kasparov product \cite{BoeijinkS,Chachich,Mesland,Zhang}. The key
reason for this utility in each case is the ability to employ
connections to explicitly write down representatives of Kasparov
products.

\addtocontents{toc}{\vspace{-6pt}}

\section{Manifold structures on spectral triples} 
\label{sec:conditions}

In this section, $(\A,\H,\D)$ will always denote a spectral triple
over a trivially graded unital $^*$-algebra $\A$, topologised as a
separable Fr\'echet algebra for which the operator norm on~$\B(\H)$ is
continuous. (In particular, its norm closure $A$ is a separable
$C^*$-algebra.)

\subsection{Spin$^c$ manifolds in NCG} 
\label{ssc:geom-cond}

We start with a discussion of some of the conditions proposed by 
Connes~\cite{ConnesGrav,ConnesRecon} to enable a reconstruction 
theorem for compact spin$^c$ (and spin) manifolds. We follow the 
setup of~\cite{Crux}, to which we refer for more detail on these
conditions.

\begin{cond}[Regularity] 
\label{cn:qc-infty}
The spectral triple $(\A,\H,\D)$ is $QC^\infty$, as set forth in
Definition~\ref{df:qc-infty}. Under our completeness assumption
on~$\A$, topologised by the seminorms~\eqref{eq:dl-snorms}, $\A$ is
then a Fr\'echet pre-$C^*$-algebra.
\end{cond}

\begin{cond}[Dimension] 
\label{cn:spec-dim}
The spectral triple $(\A,\H,\D)$ is $\Zz_p$-summable for a fixed
positive \textit{integer}~$p$. Thus, if $\Trw$ denotes any Dixmier
trace%
\footnote{We may choose and fix a particular Dixmier trace $\Trw$, 
since none of our results depend on its choice.}
corresponding to a Dixmier limit~$\om$, the linear functional
$\psi_\om(a) := \Trw(a\,\Dreg^{-p})$ is defined (and positive) 
on~$\A$. Indeed, since every $[\D,a]$ is bounded, the same formula
extends $\psi_\om$ to a positive linear functional on~$\CDA$.
\end{cond}

\begin{cond}[Finiteness and Absolute Continuity] 
\label{cn:finite}
The dense subspace $\H_\infty := \Domoo \D$ of~$\H$ is a finitely
generated projective left $\CDA$-module, and the functional $\psi_\om$
is faithful on~$\CDA$ with $\H = L^2(\H_\infty,\psi_\om)$.
\end{cond}

Keeping always the abbreviation $\Cc \equiv \CDA$, we can therefore
identify $\H_\infty \isom \Cc^m q$ where $q \in M_m(\Cc)$ is a
selfadjoint idempotent. As a left $\Cc$-module, $\H_\infty$ has a set
of generators $\xi_1,\dots,\xi_m$, and the hermitian pairing of
$\xi = \sum_{j=1}^m u_j \xi_j = \sum_{j,k=1}^m u_j q_{jk} \xi_k$ and 
$\eta = \sum_{j=1}^m v_j \xi_j = \sum_{j,k=1}^m v_j q_{jk} \xi_k$ is 
given by $\pairL_\Cc(\xi|\eta) := \sum_{j,k=1}^m u_j q_{jk} v_k^*$.

The identification of $\H$ with $L^2(\H_\infty,\psi_\om)$ means that 
the scalar product on~$\H$ is given by
\begin{equation}
\scal<\eta|\xi> = \psi_\om \bigl( \pairL_\Cc(\xi|\eta) \bigr)
= \Trw \bigl(  \pairL_\Cc(\xi|\eta) \,\Dreg^{-p} \bigr).
\label{eq:abs-cont} 
\end{equation}

\begin{rmk} 
\label{rk:left-to-right} 
The left-to-right switch of the vectors $\eta$ and~$\xi$
in~\eqref{eq:abs-cont} just takes into account that the inner product
$\pairL_\Cc(\cdot|\cdot)$ on the left $\A$-module $\H_\infty$ is
linear in the first variable. Compare the proof of
Proposition~\ref{pr:lin-bdd} above, which deals with right hermitian
modules.
\end{rmk}

\begin{rmk} 
\label{rk:cda-mod}
The reason for asking for $\H_\infty$ to be a finitely generated
projective $\CDA$-module, as opposed to $\A$-module, is to control the
representation of~$\CDA$. Classically, $\CDA$ is the Clifford algebra
determined by some Riemannian metric, and various features of
manifolds can be characterised in terms of the representation of this
algebra. For example, a manifold is spin$^c$ if and only if there is a
bundle (of spinors) providing a Morita equivalence between the
Clifford algebra and the algebra of functions. Later we shall also
look at the representation of the Clifford algebra on the bundle of
exterior forms.
\end{rmk}

\begin{rmk} 
\label{rk:first-order}
By Proposition \ref{pr:first-order}, if $(\A,\H,\D)$ satisfies the
previous three conditions and if $\H_\infty$ is (also) finite
projective over $\A$, we can amplify its algebra to $\A \ox \B$, where
$\B$ is the algebra of bounded operators mapping $\H_\infty$ to
itself, and commuting with the whole algebra $\CDA$ of
Definition~\ref{df:best-algebra} ---and the grading, if there is one.
Then $(\A \ox \B, \H, \D)$ is a spectral triple satisfying the first
order condition, and is $QC^\infty$ for the algebra~$\A$.

It is also $\Zz_p$-summable and the absolute continuity
\eqref{eq:abs-cont} implies that $\Trw(w\,\Dreg^{-p}) > 0$ whenever
$w \in \CDA$ is a nonzero positive operator on~$\H$. This condition
determines $p$ uniquely; we then say that the critical summability
parameter $p$ is the \textit{spectral dimension} of $(\A,\H,\D)$ or of
$(\A \ox \B, \H, \D)$.
\end{rmk}

\begin{cond}[Orientability] 
\label{cn:orient}
The spectral triple $(\A,\H,\D)$, of spectral dimension~$p$, is called
\textit{orientable} if there exists a Hochschild $p$-cycle
\begin{equation}
\cc = \tsum_{\al=1}^n a_\al^0 \ox a_\al^1 \oxyox a_\al^p \in Z_p(\A,\A)
\label{eq:Hoch-cycle} 
\end{equation}
such that the bounded operator
$$
C \equiv \pi_\D(\cc)
:= \tsum_\al a_\al^0 \,[\D,a_\al^1] \dots [\D,a_\al^p]
$$
is invertible and selfadjoint, has square $C^2 = 1$, and
satisfies
$$
C \D - (-1)^{p-1} \D C = 0;  \word{and}
C\,a - a\,C = 0 \word{for} a \in \A.
$$
\end{cond}

\begin{rmk} 
\label{rk:Hoch-cycle}
The Hochschild class $[\cc] \in HH_p(\A)$ may be called the
\textit{orientation} of $(\A,\H,\D)$, and we call $(\A,\H,\D,\cc)$ an
\textit{oriented} spectral triple.

One could ask for an orientation cycle $\cc \in Z_p(\A, \A \ox \B)$,
where the first tensor factors in~\eqref{eq:Hoch-cycle} can be
taken in $\A \ox \B$. This case arises when the algebra and Hilbert
space are finite dimensional (where the ordinary trace replaces the
Dixmier trace and one sets $p = 0$); see, for instance,
\cite{Krajewski,PaschkeS}. We will not treat this more general
definition here.
\end{rmk}

\begin{rmk} 
\label{rk:volume-grading}
When $p$ is even, the selfadjoint unitary operator $C = \pi_\D(\cc)$
commutes with $\A$ and anticommutes with~$\D$; in particular, it
cannot be trivial. Thus $C$ is a $\Z_2$-\textit{grading} operator,
making $(\A,\H,\D)$ an \textit{even} spectral triple. (However, this
need not coincide with some previous defined grading~$\Ga$, if the
spectral triple is already even.)

On the other hand, if $p$ is odd, then $C$ commutes with both $\A$
and~$\D$, and is thus a central element of~$\CDA$.
\end{rmk}

The basic property of the space of smooth spinors on a spin$^c$
manifold~\cite{Plymen} is encoded in the following condition, in 
which the regularity and finiteness conditions are assumed.

\begin{cond}[Spin$^c$] 
\label{cn:pmorita}
The subspace $\H_\infty$ is a \textit{pre-Morita equivalence bimodule}
between $\CDA$ and~$\A$. That is to say: $\H_\infty$ is already a left
module for the given action of $\CDA$ on~$\H$, and there is a
commuting right action of~$\A$ on $\H_\infty$, with hermitian
structures satisfying Definition~\ref{df:Morita-bimod}.
\end{cond}

\begin{rmk} 
\label{rk:commuting-alg}
The spin$^c$ condition essentially fixes the ``commuting algebra''
$\B$ of \eqref{eq:first-order} to be $\A^\op$. Indeed,
Proposition~\ref{pr:pre-Morita} establishes that the $C^*$-completion
of $\B$ is~$A^\op$; and certainly $\A^\op \subseteq \B$. Any remaining
ambiguity concerns smoothness of the ``commuting algebra''.
\end{rmk}

\begin{rmk} 
\label{rk:extra-conds}
Some further properties of spectral triples are listed 
in~\cite{Crux}. We briefly mention them here, though they are not 
essential to our present purposes. The ``first-order condition'' is 
here absorbed by our Proposition~\ref{pr:first-order}. Other possible 
requirements are as follows. 
\vspace{-9pt}
\begin{enumerate}
\item[(a)]
\textit{Closedness}:
Assuming Conditions~\ref{cn:qc-infty}--\ref{cn:orient}, for any
$a_1,\dots,a_p \in \A$ and $b \in \B$, the vanishing relation
$\Tr_\om(\Ga\,[\D,a_1]\dots [\D,a_p]\,b\,\Dreg^{-p}) = 0$ holds.
\item[(b)]
\textit{Connectivity}:
There is an orthogonal family of projectors $p_j \in \A$ such that
$\sum_j p_j = 1$ and $[\D,a] = 0$ for $a \in \A$ if and only if
$a = \sum_j \la_j p_j$ for some scalars~$\la_j \in \C$.
\item[(c)]
\textit{Reality}:
Assuming Conditions~\ref{cn:qc-infty}--\ref{cn:pmorita}, there is an 
antiunitary operator $J \: \H \to \H$ such that
$J a^* J^{-1} = a^\op$ for all $a \in \A$ (i.e., $J (\cdot)^* J^{-1}$ 
exchanges the left and right actions of $a$ by bounded operators 
on~$\H$). Moreover, $J^2 = \pm 1$, $J \D J^{-1} = \pm\D$ and also
$J \Ga J^{-1} = \pm\Ga$ in the even case; these signs depend only on
$p \bmod 8$ and coincide with those of the charge conjugation on spin 
manifolds ---see \cite{ConnesGrav} or \cite[Sect.~9.5]{Polaris}.
\end{enumerate}
\end{rmk}

\begin{defn} 
A \textit{noncommutative oriented spin$^c$ manifold} is an oriented
spectral triple $(\A,\H,\D,\cc)$ satisfying Conditions
\ref{cn:qc-infty} to~\ref{cn:pmorita}. It may be called \textit{spin},
rather than spin$^c$, if it also has the reality property.
\end{defn}

By Remark~\ref{rk:commuting-alg}, a $p$-dimensional noncommutative
oriented spin$^c$ manifold $(\A,\H,\D,\cc)$ defines a spectral triple
$(\A \ox \A^\op, \H, \D)$ which is $p$-dimensional. 

We now want to characterise the inner product on~$\H_\infty$. To do
this we need to assume that the only operators on $\H$ commuting with
both $\A$ and $\D$ are scalars. (This can be weakened by assuming
instead the connectivity condition above; then we can run the
following proof on each connected piece to get a similar statement.)

\begin{prop} 
\label{pr::matr-pairing}
Let $(\A,\H,\D)$ be a spectral triple satisfying Conditions
\ref{cn:qc-infty}, \ref{cn:spec-dim} and~\ref{cn:finite}. Suppose
moreover that $\H_\infty$ is finite projective over~$\A$, so that
$\H_\infty \isom p \A^n$ for some projector $p \in M_n(\A)$. Assume
also that $\A$ commutes with~$p$ and that only scalars commute with
all $a \in \A$ and~$\D$. Then any right Hermitian pairing on
$\H_\infty \isom p \A^n$ satisfying \eqref{eq:abs-cont} is a positive
multiple of the standard one.
\end{prop}

\begin{proof}
{}From Lemma~\ref{lm:matr-pairing}, there is a positive invertible
element $r \in p M_n(\A) p$ such that the given pairing is of the form
\eqref{eq:matr-pairing}, i.e., 
$\pairR(\xi|\eta)_r \equiv \sum_{j,k} a_j^* r_{jk} b_k$ for
$\xi = \sum_j \xi_j a_j$, $\eta = \sum_k \xi_k b_k$ in~$\H_\infty$.
If $a \in \A$, then since $ap = pa$ we get 
$a^*\xi = \sum_j \xi_j a^* a_j$ and
$a\eta = \sum_k \xi_k a b_k \in \H_\infty$. The
formula~\eqref{eq:abs-cont} then implies
$$
\braket{\xi}{a\eta} = \braket{a^*\xi}{\eta}
= \Trw \bigl( \pairR(a^*\xi|\eta)_r \,\Dreg^{-p} \bigr)
= \Trw \bigl( \pairR(\xi|r^{-1} a r\eta)_r \,\Dreg^{-p} \bigr)
= \braket{\xi}{r^{-1} a r\eta}.
$$
Hence $[r,a] = 0$ for all $a \in \A$. 

Now since $\D$ is a selfadjoint operator on~$\H$, we obtain, for
$\xi,\eta \in \H_\infty$:
\begin{align}
0 &= \braket{\D\xi}{\eta} - \braket{\xi}{\D\eta}
= \Trw \bigl( \bigl( \pairR(\D\xi|\eta)_r 
- \pairR(\xi|\D\eta)_r \bigr) \,\Dreg^{-p} \bigr)
\nonumber \\
&= \Trw \bigl( \bigl( \pairR(\D\xi|r\eta)_\A
- \pairR(\xi|r\D\eta)_\A \bigr) \,\Dreg^{-p} \bigr)
=: \bbraket{\D\xi}{r\eta} - \bbraket{r\xi}{\D\eta}
\label{eq:comm-quadform} 
\end{align}
where $\bbraket{\xi}{\eta} := \braket{r^{-1}\xi}{\eta}$ defines a new
Hilbert space scalar product. Since $r^{-1}$ is bounded with bounded
inverse, this new scalar product $\bbraket{\cdot}{\cdot}$ is
topologically equivalent to the old one $\braket{\cdot}{\cdot}$, and
so $\H$ coincides with the completion of $\H_\infty$ with respect to
either scalar product.

Now the right hand side of \eqref{eq:comm-quadform} is the quadratic
form defining the commutator $[\D,r]$ with respect to the scalar
product $\bbraket{\cdot}{\cdot}$. It vanishes on $\H_\infty$ and thus
$[\D,r] = 0$. The irreducibility condition now implies that $r$ is (a
positive multiple of) the identity $p$ in $p M_n(\A) p$, represented
by a scalar operator on~$\H$.
\end{proof}

\begin{prop} 
\label{prop:spin-c-works}
Let $(\A,\H,\D,\cc)$ be a noncommutative oriented spin$^c$ manifold
such that only scalars commute with all $a\in\A$ and $\D$. Then 
$\H_\infty$ is  finite projective as both a left and a right
$\A$-module, and $\CDA$ is finite projective as a left or right
$\A$-module. Moreover, the relations
$$
\scal<\eta|\xi> = \psi_\om \bigl( \pairR(\eta|\xi)_\A \bigr)
= \Trw \bigl(  \pairR(\eta|\xi)_\A \,\Dreg^{-p} \bigr)
$$
hold for all $\xi,\eta\in\H_\infty$. In particular, $\psi_\om$ is
faithful on~$\A^\op$.
\end{prop}

\begin{proof}
Let $\H_\infty = \Cc^m q$, as in Condition~\ref{cn:finite}. Then by
Condition~\ref{cn:pmorita} and Proposition~\ref{pr:pre-Morita}, we
also get $\H_\infty \isom p\A^n$ and $\Cc = p M_n(\A) p$. Since
$\A \subset \Cc = \CDA$, $\A$ commutes with the projector~$p$.

That property ensures that the partial trace
$\tr\: \Cc \to \A$, defined on $T = [t_{ij}] \in p M_n(\A) p$ by
$\tr(T) := \sum_{i=1}^n t_{ii}$, is a well-defined operator-valued
weight. It also shows that $\Cc$ is finite projective as a left or
right $\A$-module, because $p M_n(\A) p \subset M_n(\A)$ is a direct
summand \emph{as an $\A$-module} precisely because $p$ commutes with
the action of~$\A$. It now follows immediately that $\H_\infty$ is a
finite projective left module over $\A$, since $\H_\infty = \Cc^m q$
is a direct summand in $\Cc^m$, which by the previous discussion is
a direct summand in $\A^{mn^2}$.

We now observe that by Proposition~\ref{pr::matr-pairing} the left
$\A$-valued inner product on $\H_\infty$ is given on
$\xi = \sum_j \xi_j a_j$ and $\eta = \sum_k \xi_k b_k$ by
$$
\pairR(\xi|\eta)_\A = \la \sum_{j,k} a_i^* p_{ij} b_j,
\word{for some} \la > 0,
$$
and we verify that
\begin{align*}
\Trw\bigl( \pairR(\xi|\eta)_\A\,\Dreg^{-p} \bigr)
&= \la \sum_{ij }\Trw\bigl( a^*_i p_{ij} b_j \Dreg^{-p} \bigr)
= \la \sum_{ij} \Trw\bigl( p_{ij} b_j a_k^* p_{ki} \Dreg^{-p} \bigr)
\\
&= \Trw\bigl( \tr\bigl( \pairL_\Cc(\eta|\xi) \bigr) \Dreg^{-p} \bigr),
\end{align*}
the $\CDA$ inner product being determined by the Morita equivalence
condition. We have used the tracial property of $\psi_\om$, which is 
due to the condition $\Dreg^{-1} \in \Zz_p$ \cite{CareyGRS1}.

As a consequence of this calculation, and of the positivity of~$\tr$,
we see that $\H \isom p L^2(\A,\psi_\om)^n$. In this picture, the
operator trace is precisely
$\Tr_\H = \Tr_{L^2(\A,\psi_\om)} \circ \tr$, at least when
restricted to~$\CDA$. So finally we obtain
\begin{equation*}
\Trw\bigl( \pairR(\xi|\eta)_\A \,\Dreg^{-p} \bigr)
= \Trw\bigl( \pairL_\Cc(\eta|\xi) \,\Dreg^{-p} \bigr)
= \scal<\xi|\eta> .
\tag*{\qedhere}
\end{equation*}
\end{proof}

\subsection{Riemannian manifolds in NCG} 
\label{ssc:basic-riem}

The basic structures which all oriented Riemannian manifolds share are
(a)~the exterior algebra of differential forms, and (b)~the
representation on forms of the Clifford algebra determined by the
metric. Other features such as the Hodge $*$-operator, the
Hodge--de~Rham operator $d + d^*$ and so on, can all be obtained from
these structures. Most importantly, the Clifford algebra is linearly
isomorphic to the algebra of differential forms. Indeed, the
differential forms provide a bimodule for the Clifford algebra. This
bimodule is always $\Z_2$-graded by parity of forms, irrespective of
whether the dimension of the manifold is even or odd. This information
is captured in the following condition.

\begin{cond}[Riemannian] 
\label{cn:Riemann}
The vector space $\H_\infty := \Domoo \D$ of~$\H$ contains a cyclic
and separating vector $\Phi$ for the action of the algebra
$\Cc = \CDA$, in the algebraic sense, that is,
$\H_\infty = \set{w\Phi : w \in \CDA}$; and $w = 0$ in~$\CDA$ if and
only if $w\Phi = 0$ in~$\H_\infty$. In particular, $\H_\infty$ is a
free left $\CDA$-module of rank one. Moreover, there exists a
Hermitian pairing $\pairL_\Cc(\cdot|\cdot)$ on $\H_\infty$ such that%
\footnote{In this subsection, for notational convenience, we write 
$\Cc = \CDA$ in all subscripts.}
\begin{equation}
\scal<\eta|\xi> = \psi_\om \bigl( \pairL_\Cc(\xi|\eta) \bigr)
= \Trw \bigl(  \pairL_\Cc(\xi|\eta) \,\Dreg^{-p} \bigr)
\word{for}  \xi,\eta \in \H_\infty;
\label{eq:abs-cont-bis} 
\end{equation}
and such that $z = \pairL_\Cc(\Phi|\Phi)$ is a strictly positive
central element of~$\CDA$. We demand that $\H_\infty$ be finite
projective as a left $\A$-module. There is also a grading
operator~$\eps$ on~$\H$ such that $\eps\Phi = \Phi$, making
$(\A,\H,\D)$ an even spectral triple. (It follows that
$\eps(\cdot)\eps$ is the parity grading of~$\CDA$.) Without loss of
generality, we assume that $\|\Phi\| = 1$.
\end{cond}

\begin{defn} 
\label{df:NC-Riemann}
A \textit{noncommutative oriented Riemannian manifold} is an oriented
spectral triple with a distinguished vector $(\A,\H,\D,\cc,\Phi)$,
satisfying Conditions~\ref{cn:qc-infty} to~\ref{cn:orient} and
Condition~\ref{cn:Riemann}.
\end{defn}

For instance, in the commutative case of a compact oriented Riemannian
manifold $(M,g)$ without boundary, we take $\A = C^\infty(M)$; $\H$ is
the Hilbert space of square-integrable differential forms of all
degrees; and $\D = d + d^*$ is the Hodge--de~Rham (or Hodge--Dirac)
operator. Here $\H_\infty$ is the space of smooth differential forms
$\Om^\8(M)$, and we take $\Phi$ to be the $0$-form (constant
function) $1$; and $\CDA$ is the complexified Clifford algebra, whose
$C^*$-completion is~$\Cliff(M)$.

\vspace{6pt}

Let us examine some of the implications of Condition~\ref{cn:Riemann}.

\begin{rmk} 
\label{rk:smooth-module}
We have already remarked, using \eqref{eq:smooth-switch}, that any
operator in $\Domoo \dl$ preserves $\H_\infty = \Domoo \D$. Thus,
under the assumption of regularity, Condition~\ref{cn:qc-infty}, all
elements of $\CDA$ map $\H_\infty$ to itself, so $\H_\infty$ is indeed
a left $\CDA$-module. 
\end{rmk}

\begin{rmk} 
\label{rk:fgp}
Unlike the spin$^c$ case, we must insert finite projectivity of
$\H_\infty$ over $\A$ by hand. Omitting this condition would mean that
we do not obtain Kasparov's fundamental class, and we would be much
more limited in the Kasparov products we could take. On the other
hand, it is interesting to speculate whether continuous trace
$C^*$-algebras over compact manifolds give (nonunital)
`Riemannian manifolds' which fail to have this finite projective
property. This is related to the work of \cite{Zhang}.
\end{rmk}

There is a positive linear functional $\sg_\Phi \: \CDA \to \C$ given
by
\begin{equation}
\sg_\Phi(w) := \psi_\om \bigl( \pairL_\Cc(w\Phi|\Phi) \bigr)
= \psi_\om \bigl( w\,\pairL_\Cc(\Phi|\Phi) \bigr) = \psi_\om(wz).
\label{eq:sigma_fn} 
\end{equation}
By equation \eqref{eq:abs-cont-bis} in Condition~\ref{cn:Riemann},
$\sg_\Phi$ is a \textit{vector state}:
$$
\sg_\Phi(w) = \scal<\Phi|w\Phi>,  \word{for all}  w \in \CDA.
$$

\begin{lemma} 
\label{lm:interchange}
Suppose that $(\A,\H,\D)$ is a spectral triple satisfying Conditions
\ref{cn:qc-infty}, \ref{cn:spec-dim} and~\ref{cn:Riemann}. Then
$\psi_\om, \sg_\Phi \: \CDA \to \C$ are faithful traces; and
$\H = L^2(\H_\infty,\psi_\om)$.
\end{lemma}

\begin{proof}
We have already remarked that $\psi_\om$ is a trace on~$\Cc = \CDA$, 
in view of the tracial properties of $\Zz_p$-summable $QC^\infty$ spectral
triples \cite{CareyGRS1,CareyRSS,CiprianiGS}.

The central element $z = \pairL_\Cc(\Phi|\Phi)$ of $\Cc$ is positive,
with $\psi_\om(z) = \|\Phi\|^2 = 1$. From~\eqref{eq:sigma_fn},
$$
\sg_\Phi(uw) = \psi_\om(uwz) = \psi_\om(wzu) = \psi_\om(wuz)
= \sg_\Phi(wu) \word{for all}  u,w \in \Cc,
$$
so that $\sg_\Phi$ is a trace, too.

Let $w \in \Cc$. Since
$\sg_\Phi(w^*w) = \scal<\Phi|w^*w\Phi> = \|w\Phi\|^2$, then
$\sg_\Phi(w^*w) = 0$ if and only if $w\Phi = 0$. This implies $w = 0$,
since $\Phi$ is separating for~$\Cc$. Hence $\sg_\Phi$ is faithful.

We recall that if $a,b,c \in B(\H)$ with $a \leq b$, then
$c^*ac \leq c^*bc$. Now notice, for $0 < w \in \Cc$, that
$$
\sg_\Phi(w) = \psi_\om(wz) = \psi_\om(w^{1/2} z w^{1/2}) \leq
\psi_\om(w^{1/2} \|z\| w^{1/2}) = \|z\| \,\psi_\om(w),
$$
since $\psi_\om$ is positive. The faithfulness of~$\sg_\Phi$ shows 
that $0 < \sg_\Phi(w) \leq \|z\| \,\psi_\om(w)$ for each positive
$w \in \Cc$ with $w \neq 0$. Hence $\psi_\om$ is faithful, too.

The equality $\H = L^2(\H_\infty,\psi_\om)$ now follows from
Definition~\ref{df:ltwo-space}, since $\H_\infty$ is a free rank-one
$\Cc$-module and $\psi_\om$ is faithful on~$\Cc$.
\end{proof}

\begin{rmk} 
\label{rk:high-fidelity}
In Condition~\ref{cn:finite}, the functional $\psi_\om$ is required 
to be faithful. The previous lemma shows that this requirement is 
redundant in the presence of Condition~\ref{cn:Riemann}.
\end{rmk}

It is immediate that $\sg_\Phi: \CDA \to \C$ has a normal extension
$\tilde\sg_\Phi$ to the double commutant $\CDA''$ given by
$$
\tilde\sg_\Phi(T) := \scal<\Phi|T\Phi> \word{for all} T \in \CDA'' .
$$

\begin{lemma} 
\label{lm:extension}
Suppose that $(\A,\H,\D)$ is a spectral triple satisfying Conditions
\ref{cn:qc-infty}, \ref{cn:spec-dim} and~\ref{cn:Riemann}.
\begin{enumerate}
\item[\textup{(a)}]
The vector $\Phi \in \H$ is cyclic and separating for the double 
commutant $\CDA''$.
\item[\textup{(b)}]
The functional $\tilde\sg_\Phi$ is a faithful normal finite trace
on~$\CDA''$.
\end{enumerate}
\end{lemma}

\begin{proof}
Ad~(a):
The dense subspace $\H_\infty$ of $\H$ equals $\CDA\Phi$ by
hypothesis, so that $\CDA''\Phi$ is also dense in~$\H$, i.e., $\Phi$
is a cyclic vector for~$\CDA''$ in the topological sense.
 
Any cyclic vector associated to a tracial vector state is separating;
as we show. Assume that $0 < T \in \CDA''$ with $T\Phi = 0$, and that
$0 < T_n \in \CDA$ converge strongly to~$T$. Then
$T_n\Phi \to T\Phi = 0$. Hence, for each $w \in \CDA$,
\begin{align*}
\|T_n^{1/2} w \Phi\|^2 
&= \scal<\Phi | w^* T_n w \Phi> = \scal<\Phi | T_n ww^* \Phi>
\\ 
&= \scal<T_n \Phi | ww^* \Phi> \leq \|T_n \Phi\|^2 \|ww^* \Phi\|^2.
\end{align*}
The second equality uses Lemma~\ref{lm:interchange}, the tracial
property of~$\sg_\Phi$. Hence $\lim_n \|T_n^{1/2} w \Phi\| = 0$. As
$\CDA\Phi$ is dense in $\H$, $T_n^{1/2}$ converges strongly to $0$.
Hence $T_n$ converges strongly to $0$. By uniqueness of strong limits
in $\CDA''$, $T = 0$. Thus $\Phi$ is separating for $\CDA''$.

Ad~(b):
Since $\scal< \Phi | T^*T\Phi > = \|T\Phi\|^2$, it follows from the
separating property of $\Phi$ on $\CDA''$ in (a) that
$\tilde{\sg}_\Phi(T^*T) = 0$ if and only if $T = 0$. Hence
$\tilde{\sg}_\Phi$ is faithful. It is straightforward to show that any
normal extension of a trace on $\CDA$ is a trace on $\CDA''$. Hence
$\tilde{\sg}_\Phi$ is a faithful trace. It is evidently finite.
\end{proof}

\begin{rmk} 
\label{rk:extended-ncint}
Since $z = \pairL_\Cc(\Phi|\Phi)$ is central in~$\CDA$, it lies
in~$\CDA'$, hence $z$~is also central in $\CDA''$. It is therefore
pertinent to ask whether the formula \eqref{eq:sigma_fn} can be
extended to the bicommutant, that is, whether there exists a (unique)
faithful normal trace $\tau\: \CDA'' \to \C$, such that
$\tilde\sg_\Phi(T) = \tau(Tz)$ for all $T \in \CDA''$, which extends
the faithful trace $\psi_\om$ on~$\CDA$, 

Such an extension to a normal trace is not trivial, since $\psi_\om$
may fail to be strongly continuous on~$\CDA$ because it arises from a
Dixmier trace~$\Trw$. It is, however, possible to construct such
a~$\tau$ provided one assumes explicitly that $z$ is invertible
in~$\CDA$. We do not require this extended trace for the present
purposes, so we leave it aside.
\end{rmk}

In view of Lemma~\ref{lm:extension}, from now on we shall write 
simply $\sg_\Phi$, rather than $\tilde\sg_\Phi$, to denote the vector 
state $\scal<\Phi|(\cdot)\Phi>$ on $\CDA''$ as well as on~$\CDA$.

As a consequence of the Tomita--Takesaki theory \cite{TakesakiTwo}
there exists an antiunitary operator $J = J_\Phi$ such that
$J (\cdot)^* J : \CDA'' \to \CDA'$ is an antiisomorphism of von
Neumann algebras and
\begin{equation}
J(T\Phi) = T^* \Phi,  \word{for all}  T \in \CDA''.
\label{eq:Tomita} 
\end{equation}
The assignment $w^\op := J w^* J$ enables a commuting right action of
the algebra $\CDA$ on~$\H$. \emph{The first order condition is
automatically satisfied}:
$$
[a, b^\op] = 0, \quad  [[\D, a], b^\op] = 0,
\word{for all} a,b \in \A.
$$

\begin{lemma} 
\label{lm:grading-Riemann}
Suppose $(\A,\H,\D,\cc,\Phi)$ satisfies Conditions \ref{cn:qc-infty},
\ref{cn:spec-dim}, \ref{cn:orient} and~\ref{cn:Riemann}. Write
$C = \pi_\D(\cc)$, let $\eps$ be the grading operator, and let 
$J$ be the Tomita conjugation~\eqref{eq:Tomita}. Then
\begin{enumerate}
\item[\textup{(a)}]
$[J,\eps] = 0$;
\item[\textup{(b)}]
$\eps = C J C J = C C^\op$  when $p$ is even;
\item[\textup{(c)}]
$\eps C = - C \eps$ and $C = J C J = C^\op$, when $p$ is odd;
\item[\textup{(d)}]
$J$ maps $\H_\infty$ to $\H_\infty$ bijectively;
\item[\textup{(e)}]
$\H_\infty$ is a pre-Morita equivalence bimodule between $\CDA$ and
itself.
\end{enumerate}
\end{lemma}

\begin{proof} 
The Tomita conjugation $J$ for the vector $\Phi$ is the unique
antilinear operator defined by $J(T\Phi) := T^*\Phi$ for all
$T \in \CDA''$. This operator is bounded and antiunitary since the
vector state $\sg_\Phi$ is a trace. Note that
$J \Phi = \Phi$ since $\CDA'' \supset \CDA$ is unital. The additional
properties that $J^2 = 1$, and that $w = J w^* J$ if and only if
$w \in \CDA$ is central, then follow from the cyclic and separating
properties of the vector~$\Phi$.

Ad~(a):
To see that $[J,\eps] = 0$, take 
$w = w_\mathrm{e} + w_\mathrm{o} \in \CDA$, where
$\eps\,w_\mathrm{e}\,\eps = w_\mathrm{e}$
and $\eps\,w_\mathrm{o}\,\eps = - w_\mathrm{o}$. Since $\eps$ is 
selfadjoint, we get $w^* = (w^*)_\mathrm{e} + (w^*)_\mathrm{o}
= (w_\mathrm{e})^* + (w_\mathrm{o})^*$. Then
$$
\eps J w\Phi = \eps w^*\Phi = (w_\mathrm{e}^* - w_\mathrm{o}^*)\Phi
= J (w_\mathrm{e} - w_\mathrm{o})\Phi = J \eps w \Phi.
$$
Hence $[J, \eps] w \Phi = 0$ for all $w \in \CDA$; equivalently,
$[J, \eps]\xi = 0$ for all $\xi \in \H_\infty$, or $\xi \in \H$ for 
that matter. Hence $[J, \eps] = 0$.

Ad~(b):
Let $p$ be even. Then $C\,\eps = \eps\,C$ as $C \in C_\D(\A)$ is of
even parity by Condition~\ref{cn:orient}. Moreover $U := C\,\eps$
commutes with $\D$ and $\A$ by Conditions \ref{cn:orient}
and~\ref{cn:Riemann}. Hence $U \in \CDA'$ and it is a selfadjoint
unitary. Note also, since $C \in \CDA''$, that $J C J \in \CDA'$ and
$\Phi$ is cyclic and separating for $\CDA'$. Then $(JCJ - U) \Phi 
= J C J \Phi - C\eps \Phi = J C \Phi - C \Phi = C \Phi - C \Phi = 0$.
Hence $U = J C J$ by the separating property of~$\Phi$. It follows 
that $\eps = C U = C J C J$.

Ad~(c):
Let $p$ be odd. Then $C$ commutes with $\D$ and $\A$ by
Condition~\ref{cn:orient}. In this case $C$ is central in $C_\D(\A)$,
which occurs if and only if $C = J C^* J = J C J$. It is immediate
that $C \eps = -\eps C$ since $C$ has odd parity by
Condition~\ref{cn:orient}.

Ad~(d):
This is clear, since $J w \Phi = w^* \Phi$ for all $w \in \CDA$.

Ad~(e):
We define a \textit{right action} of $\Cc = \CDA$ on $\H_\infty$ by
$$
w^\op \xi \equiv \xi \cdot w := J w^* J\xi,
\word{for all}  \xi \in \H_\infty, \ w \in \Cc.
$$
Note that, in particular
$w \Phi = J(w^*\Phi) = w^\op J \Phi = w^\op \Phi = \Phi \cdot w$
for all $w \in \Cc$. The right-module inner product on $\H_\infty$ is
defined by 
$$
\pairR(\xi|\eta)_\Cc := \pairL_\Cc(J\xi|J\eta),
$$
which works because $J$ is antilinear. We now check that this agrees
with the natural hermitian pairing of (compact) endomorphisms of
$\H_\infty$ as a left $\Cc$-module, in order that $\H_\infty$ satisfy 
the requirements of a pre-Morita equivalence bimodule.

For $\xi,\eta,\zeta \in \H_\infty$ we write $\xi = w\,\Phi$,
$\eta = v\,\Phi$, $\zeta = u\,\Phi$ where $w,v,u \in \Cc$. The
relationship between the inner products follows from the centrality of
$z = \pairL_\Cc(\Phi|\Phi)$ and the computation
\begin{align*}
\xi\,\Theta_{\eta,\zeta}
&=  \pairL_\Cc(\xi|\eta)\,\zeta = \pairL_\Cc(w\Phi|v\Phi)\,u\Phi
= w z v^* u \Phi = w (v^* z u) \Phi 
\\
&= w (v^* z u)^\op \Phi = (v^* z u)^\op w \Phi
= w\Phi \cdot \pairL_\Cc(v^*\Phi|u^*\Phi)
= \xi \cdot \pairL_\Cc(J\eta|J\zeta).
\end{align*}
Hence every rank-one operator is contained in the range of the inner
product and in the range of the
right action of $\CDA$. 

The proof is completed by checking the estimates
\eqref{eq:Morita-norms}. For the right-module inner product, we get, 
for all $\xi = w\Phi \in \H_\infty$ and $v \in \Cc$:
\begin{align*}
\pairR(v\xi | v\xi)_\Cc
&= \pairL_\Cc(Jvw\Phi | Jvw\Phi) = w^* v^* z v w
= z^{1/2}w^*v^*v w z^{1/2}
\\
&\leq z^{1/2}w^* \|v^*v\| w z^{1/2} 
= \|v\|^2 \, w^* z w
=  \|v\|^2 \, \pairL_\Cc(Jw\Phi|Jw\Phi)
=  \|v\|^2 \, \pairR(\xi|\xi)_\Cc \,.
\end{align*}
The estimate for the left-module inner product is similar.
\end{proof}

In odd dimensions the centrality of $C = JCJ$ gives us a second
$\Z_2$-grading.

\begin{corl} 
\label{cr:odd-big-algebra}
Suppose $(\A,\H,\D)$ satisfies Conditions \ref{cn:qc-infty},
\ref{cn:spec-dim}, \ref{cn:orient} and~\ref{cn:Riemann}, and that the
spectral dimension $p$ is an odd integer. Then
$P_\pm := \half(1 \pm C)$ are two complementary projectors commuting
with both actions of~$\CDA$, thereby determining a second $\Z_2$-grading
on $\CDA = \CDA^+ \oplus \CDA^-$ where $\CDA^+ := P_+\,\CDA\,P_+$
and $\CDA^- := P_-\,\CDA\,P_-$. Moreover $\CDA^+=\eps\CDA^-\eps$.
\end{corl}

Finally we reconcile the two descriptions of the Hilbert space coming
from Conditions~\ref{cn:Riemann} and~\ref{cn:finite}. Since 
Lemma~\ref{lm:interchange} shows that $\psi_\om$ is faithful on
$\A \subset \CDA$, assuming Condition~\ref{cn:Riemann} of course, that
part of Condition~\ref{cn:finite} is redundant.

\begin{corl} 
\label{cr:op-weight}
Suppose $(\A,\H,\D)$ satisfies Conditions \ref{cn:qc-infty},
\ref{cn:spec-dim}, and~\ref{cn:Riemann}. Then $(\A,\H,\D)$ satisfies
Condition~\ref{cn:finite} if and only if:
\begin{enumerate}
\item[\textup{(i)}]
$\CDA$ is a finite projective left $\A$-module (treating
$\A \subset \CDA$);
\item[\textup{(ii)}]
there is some operator-valued weight $\Psi\: \CDA \to \A$ such that
$\psi_\om = \psi_\om \circ \Psi$ for all Dixmier limits~$\om$.
\end{enumerate}
\end{corl}

\begin{proof}
($\Rightarrow$): 
By Lemma~\ref{lm:op-weight}, there is an operator valued weight
$\Psi\: \CDA \to \A$ such that 
$\pairL_\A(\xi|\eta) = \Psi\bigl( \pairL_\Cc(\xi|\eta) \bigr)$ for all
$\xi,\eta \in \H_\infty$. Then Condition~\ref{cn:Riemann} together with
(i) and~(ii) imply that $\H_\infty$ is a finite projective left
$\A$-module and $\scal<\eta|\xi> = \psi_\om(\pairL_\A(\xi|\eta))$.

($\Leftarrow$):
Conditions \ref{cn:finite} and~\ref{cn:Riemann} together imply that
$\CDA \Phi$, and also $\CDA$ by $\A$-linear isomorphism, are finitely 
generated and projective over~$\A$ ---the existence of~$\Psi$ then 
follows from Lemma~\ref{lm:op-weight}--- and that
$\psi_\om(\Psi(\pairL_\Cc(\xi|\eta)))
= \psi_\om(\pairL_\Cc(\xi|\eta))$, for $\xi,\eta \in \H_\infty$.
Fullness of the hermitian pairing on~$\CDA$ and norm continuity of
$\psi_\om$ on~$\CDA$ then entails $\psi_\om \circ \Psi = \psi_\om$.
\end{proof}

\subsection{Kasparov's fundamental class} 
\label{ssc:kas-mod}

In \cite[Defn.-lemma~4.2]{KasparovEqvar}, Kasparov showed that a
Riemannian manifold $(M,g)$ has a fundamental class
$\la \in KK(C(M) \ox \Cliff(M),\C)$. The Kasparov product with this
class provides isomorphisms \cite[Thms.~4.9, 4.10]{KasparovEqvar}:
\begin{align*}
- \hatox_{C(M)} \la 
&: K_*(C(M)) = KK(\C, C(M)) \to KK(\Cliff(M), \C),
\\
- \hatox_{\Cliff(M)} \la
&: KK(\C,\Cliff(M)) \to KK(C(M), \C) = K^*(C(M)).
\end{align*}
Here we have been careful to use $KK$ notation (rather than
just $K$-homology or $K$-theory notation) where the algebra is
regarded as $\Z_2$-graded.

The class $\la$ can be represented by an unbounded Kasparov module
\cite{BaajJ} given by the Hilbert space $\H = L^2(\La^\8 T^*_\C M,g)$,
the Hodge--de Rham operator $d + d^*$, and the representation of
$C(M)$ given by multiplication operators. The representation of the
Clifford algebra $\Cliff(M)$ is given as follows (see
\cite[Appendix~A]{Crux} for a similar discussion).

Observe that $\D = d + d^*$ is the Dirac operator associated to the
left action of the Clifford algebra and the Levi-Civita connection.
Let $\D'$ be the Dirac operator associated to the right action of
the Clifford algebra, and let $\eps$ denote the grading of forms by
degree. On $k$-forms in $\La^k T^*_\C M$, where 
$\eps|_{\La^k} = (-1)^k$, the operator $\D'$ is given by
$$ 
\D' \bigr|_{\La^k} := (-1)^k (d - d^*).
$$
Let $\Dtilde := i\D'\eps = i(d - d^*)$. One can check (see, for
instance, \cite[Sect.~9.B]{Polaris} or the arguments below) that
$df \mapsto [\Dtilde, f]$ provides a representation of the algebra of
sections of the (real) Clifford algebra bundle for the quadratic form
$-g$ on $T^*M$.

Passing to the complexification, we get a representation of the
complex Clifford algebra $\Cliff(M)$, which graded-commutes with the
left action of the Clifford algebra, and so graded-commutes with the
symbol of $d + d^*$. Standard techniques, as described in
\cite[Defn.-lemma~4.2]{KasparovEqvar}, now show that there is a
graded, even Kasparov module for $C(M) \ox \Cliff(M)$, with this
representation of the Clifford algebra.

This complicated construction is necessary for the following reason.
The right action of $\Cliff(M)$ has bounded commutators with $d + d^*$
because their principal symbols commute. However, the right action
does not commute with the grading; rather, it graded-commutes, so we
need to employ graded commutators, which are not bounded. Thus to get
an honest Kasparov module for the Clifford algebra we must construct
this new representation.%
\footnote{Thanks to Nigel Higson for explaining the solution of this
conundrum to us.}

Concretely, the standard left Clifford action of a $1$-form $\al$ on 
forms, coming from~$\D$, is given by
$c(\al) := \eps(\al) - \iota(\al^\3)$, where 
$\eps(\al) \: \om \mapsto \al \w \om$ is exterior multiplication, and
$\iota(\al^\3)$ is contraction with the $g$-dual vectorfield $\al^\3$
of~$\al$. One checks that the the representation coming from~$\D'$ is 
given for $1$-forms by $c'(\bt) := \eps(\bt) + \iota(\bt^\3)$. The 
graded commutation $[c(\al), c'(\bt)]_+ = 0$ is now immediate.

What we will now show is that our characterisation of the Riemannian
structure of a manifold allows for the construction of an analogous
class, even in the noncommutative case.

\begin{defn} 
\label{df:dee-right}
Let $(\A,\H,\D,\cc,\Phi)$ be a noncommutative Riemannian manifold. We
define $\D' := J\D J$, and set $\Dtilde := i\D'\eps$. As before, we
write $b^\op = Jb^*J$ for the right action of~$\A$.
\end{defn}

\begin{lemma} 
\label{lm:anti-comm}
Let $(\A,\H,\D,\cc,\Phi)$ be a noncommutative Riemannian manifold, and
let $a,b \in \A$. Then
\begin{equation}
[\D, a]\,[\Dtilde, b^\op] + [\Dtilde, b^\op]\,[\D, a] = 0.
\end{equation}
\end{lemma}

\begin{proof}
This is an easy computation. First note that
$[\Dtilde, b^\op] = iJ[\D, b^*]J\eps$. Hence
$$
[\D, a]\,[\Dtilde, b^\op] + [\Dtilde, b^\op]\,[\D, a] 
= i\bigl( [\D,a]J[\D,b^*]J - J[\D,b^*]J[\D,a] \bigr)\eps = 0,
$$
since $J \CDA J \subset \CDA'$.
\end{proof}

\begin{lemma} 
\label{lm:bounded-acomm}
Let $(\A,\H,\D,\cc,\Phi)$ be a noncommutative Riemannian manifold. If
$b \in \A$, then the linear map $T_b \: \H_\infty \to \H_\infty$
defined by
$$ 
T_b := \D\,[\Dtilde, b^\op] + [\Dtilde, b^\op]\,\D
$$
extends to a bounded operator on~$\H$.
\end{lemma}

\begin{proof}
Since $[\Dtilde, b^\op]$ commutes with the left action of~$\A$, the
commutator $[T_b, a]$ vanishes for each $a \in \A$:
$$
[T_b, a] = [\D, a]\,[\Dtilde, b^\op] + [\Dtilde, b^\op]\,[\D, a] = 0,
$$
by Lemma \ref{lm:anti-comm}; thus $T_b$ is $\A$-linear on $\H_\infty$.
Since $\H_\infty$ is finitely generated and projective over~$\A$, by
Condition \ref{cn:finite}, $T_b$ extends to a bounded operator on $\H$
by Proposition~\ref{pr:lin-bdd}.
\end{proof}

\begin{rmk} 
\label{rk:commutator-estimate}
We have not assumed that $J$ has bounded commutator with $|\D|$ ---or
equivalently, with~$\Dreg$. However, if $[\Dreg,J]$ is bounded, we
also get additional smoothness. Indeed, in that case, for each
$b \in \A$, the commutator
$$
[\Dreg, [\Dtilde, b^\op]]
= iJ [\Dreg, [\D, b^*]] J \eps + i\,[\Dreg,J]\,[\D,b^*]J\eps
- i\,J[\D,b^*]\,[J,\Dreg]\eps
$$
is bounded, since $[\D, b^*] \in \Dom \tilde\dl$. In fact, on 
replacing $\Dreg$ by $\Dreg^k$ or~$|\D|^k$ for any~$k$, we see that
$[\Dtilde, b^\op] \in \Domoo \dl$.
\end{rmk}

Even without the additional smoothness that comes from the boundedness
of $[|\D|, J]$, we can prove the existence of Kasparov's Riemannian
$KK$-class in the present context.

\begin{prop} 
\label{pr:fun-class}
Let $(\A,\H,\D,\cc,\Phi)$ be a noncommutative Riemannian manifold with
spectral dimension~$p$. Then the triple
$$
(\A \ox \CDA^\op, \H, \D)
$$
is a $\Z_2$-graded spectral triple (with $\A \ox \A^\op$ even)
satisfying the first order condition, and represents a class
$\la \in KK(A \ox C^\op, \C)$, where $C^\op$ is the norm closure of
$\CDA^\op$, regarded as a $\Z_2$-graded $C^*$-algebra.
\end{prop}

\begin{proof}
The bounded selfadjoint operator $F_\D = \D\,\Dreg^{-1}$ anticommutes
with the grading $\eps$. Also, $1 - F_\D^2 = (1 + \D^2)^{-1}$ is
compact, so we need only check that the graded commutators of~$F_\D$
with elements of $\A \ox \CDA^\op$ are compact. If instead we consider
an element $y \in \A \ox \Cc_{\Dtilde}(\A^\op)$, we get
$$
[F_\D, y]_\pm = [\D, y]_\pm \,\Dreg^{-1} + \D\,[\Dreg^{-1}, y]_\pm.
$$
The first summand is compact by the compactness of $\Dreg^{-1}$ and
the boundedness of graded commutators with~$\D$, which follows from
Lemma~\ref{lm:bounded-acomm}. The second term requires more care. The
argument of Lemma~2.3 of~\cite{CareyP1}, lightly adjusted to handle
graded commutators, shows that if $[\D,y]_\pm$ is bounded, then for
any real $t > 0$, we get
$$
[y, (t + \D^2)^{-1}]_\pm
= \D(t + \D^2)^{-1} [\D, y]_\pm (t + \D^2)^{-1}
\mp (t + \D^2)^{-1} [\D, y]_\pm \D(t + \D^2)^{-1}.
$$
Now we can employ \cite[Prop.~2.4]{CareyP1} to show that
$\D[(1 + \D^2)^{-1/2}, y]_\pm$ is compact, and in fact lies in
$\L^s(\H)$ for all $s > p$.

It remains to show that $C^\op$ is isomorphic to the norm closure of
$\Cc_{\Dtilde}(\A^\op)$. We define an algebra homomorphism
$\al\: \CDA^\op \to \Cc_{\Dtilde}(\A^\op)$ on generators by
$$
\al(a^\op) := a^\op,  \qquad
\al\bigl( [\D,a]^\op \bigr) := [\D,a]^\op(-i\eps) = [\Dtilde, a^\op].
$$
One checks that $\al$ is $*$-preserving on generators, so it extends
to a $*$-isomorphism between the norm closures of $\CDA^\op$ and
$\Cc_{\Dtilde}(\A^\op)$. Thus we regard the action of $\CDA^\op$ as
being via the representation in $\Cc_{\Dtilde}(\A^\op)$.
\end{proof}

\begin{eg} 
\label{eg:Kasparov-class}
In the classical case of the Clifford algebra acting on the left of
the differential forms over a closed $C^\infty$ manifold $M$, the
$\Z_2$-graded spectral triple constructed in
Proposition~\ref{pr:fun-class} yields precisely Kasparov's fundamental
class.
\end{eg}

\begin{rmk} 
\label{rk:something-else}
Given $a \in \A$, the operator
$F_\D\,[\Dtilde,a^\op] + [\Dtilde,a^\op]\,F_\D$ is compact. Using
$[\Dtilde, a^\op] = -i[\D,a]^\op \eps$ and $F_\D \eps = -\eps F_\D$,
we find that the operator
$-i\bigl( F_\D\,[\D,a]^\op - [\D,a]^\op\,F_\D \bigr)\eps$ is also
compact. This mirrors exactly what we see in the classical case, and
in that context may be traced to the commuting of the principal
symbols of the pseudodifferential operators $F_\D$ and $[\D,a]^\op$.
\end{rmk}

\addtocontents{toc}{\vspace{-6pt}}

\section{The main theorems: from spin$^c$ to Riemannian and back} 
\label{sec:main-thms}

The existence of the fundamental class $\la$ for a noncommutative
Riemannian manifold allows us to ask about Poincar\'e duality in this
picture. In the spin$^c$ and spin settings, Poincar\'e duality has
been considered as one of the defining, or at any rate desirable,
properties of noncommutative manifolds
\cite{Book,ConnesGrav,ConnesRecon,Crux}.

For noncommutative spin$^c$ manifolds, this duality has the following
formulation. As before, $A$ and $C$ will denote the $C^*$-algebras
obtained by taking the respective norm closures of $\A$ and
$\Cc = \CDA$.

\begin{cond}[Spin$^c$ Poincar\'e duality] 
\label{cn:spinc-pdual}
The class $\mu \in KK^p(A \ox A^\op, \C)$ of the $p$-dimensional
noncommutative spin$^c$ manifold $(\A,\H,\D,\cc)$ determines for each
$j = 0,1$ an isomorphism
$$
- \ox_A \mu : KK^j(\C,A) \to KK^{j+p}(A^\op,\C) \isom KK^{j+p}(A,\C).
$$
\end{cond}

In the Riemannian but not necessarily spin$^c$ setting, we may
formulate it as follows.

\begin{cond}[Riemannian Poincar\'{e} duality] 
\label{cn:riem-pdual}
The class $\la \in KK(A\ox C^\op, \C)$ of the noncommutative
Riemannian manifold $(\A,\H,\D,\cc,\Phi)$ defines two isomorphisms for
each $j = 0,1$:
$$
- \ox_{C^\op} \la : KK^j(\C, C^\op) \to KK^j(A, \C),  \qquad
- \ox_A \la : KK^j(\C, A) \to KK^j(C^\op, \C).
$$
\end{cond}

We now state and prove two theorems that produce the desired 
isomorphisms in each case.

To handle the even and odd cases together, we adopt the following
notational convention. The algebra $\Cliff_1^p$ denotes $\Cliff_1$ if
$p$ is an odd integer and $\C$ if $p$ is even. Similarly, the vector
space $\C^2_p$ will denote $\C^2$ if $p$~is odd and $\C$ if $p$~is
even. Thus
$$
\A \ox \Cliff_1^p = \begin{cases}
\A \ox \Cliff_1 &\text{for $p$ odd},  \\ 
\A &\text{for $p$ even},  \end{cases}  \qquad
\H \ox \C^2_p \isom \begin{cases}
\H \ox \C^2  &\text{for $p$ odd},  \\ 
\H  &\text{for $p$ even}.  \end{cases}
$$
We also observe that the transpose map gives a $*$-isomorphism from
$\Cliff_1 \isom M_2(\C)$ to its opposite. This notation will
streamline our discussion of the odd and even cases. To better
identify the classes obtained from Kasparov products, some details 
about odd Kasparov classes are laid out in Appendix~\ref{app:KK-odd}.

\begin{thm} 
\label{th:spin-riem}
Let $(\A,\H,\D,\cc)$ be a $p$-dimensional noncommutative spin$^c$
manifold with Kasparov class
$\mu \in KK^p(A\ox A^\op,\C) \isom KK(A\ox A^\op \ox \Cliff_1^p,\C)$.
Regard the conjugate module $(\H_\infty \ox \C^2_p)^\3$ as an
$(\A \ox \Cliff_1^p)$-$\CDA$-bimodule, graded in odd spectral
dimensions, with class $\sg \in KK(C^\op, A^\op \ox \Cliff_1^p)$. Then
$\la := \sg \ox_{A^\op\ox\Cliff_1^p} \mu \in KK(A \ox C^\op, \C)$ is
the class of a noncommutative Riemannian manifold. If $\mu$ satisfies
spin$^c$ Poincar\'e duality, then $\la$ satisfies Riemannian
Poincar\'e duality.
\end{thm}

\begin{thm} 
\label{th:riem-spin}
Let $(\A,\H,\D,\cc,\Phi)$ be a noncommutative Riemannian manifold and
let $\la \in KK(A \ox C^\op, \C)$ be its Kasparov class. Let $\E$ be a
pre-Morita equivalence bimodule between
either $\CDA$ or $\CDA^+$ (according to parity of the spectral 
dimension) and~$\A$, with class
$\tau := [\E \ox \C^2_p] \in KK(A^\op \ox \Cliff_1^p, C^\op)$. Then 
$\mu := \tau \ox_{C^\op} \la \in KK(A \ox A^\op \ox \Cliff_1^p, \C)$
is the class of a noncommutative spin$^c$ manifold. If $\la$ satisfies
Riemannian Poincar\'e duality, then $\mu$ satisfies spin$^c$
Poincar\'e duality.
\end{thm}

\begin{rmk} 
\label{rk:bait-and-switch}
As in Proposition~\ref{pr:married-triple}, a few left-right issues
must be addressed in order to formulate the Kasparov product
correctly. However, just as in that Proposition, we can unpack the
definitions to find that, for example,
$$
\E \ox_{C^\op} \H_\infty \isom \H_\infty \ox_C \E,
$$
where we regard $\E$ on the left hand side as a right $C^\op$-module
and on the right hand side as a left $C$-module.
\end{rmk}

\subsection{Proof of Theorem \ref{th:spin-riem}} 

We prove the even case first; the odd case needs only a few 
modifications, to be discussed later.

We begin with the noncommutative spin$^c$ manifold $(\A, \H,\D,\cc)$
and the pre-Morita equivalence bimodule $\H_\infty$ between $\CDA$
and~$\A$ given by the spin$^c$ condition. As in
Proposition~\ref{prop:spin-c-works}, we identify $\H_\infty = q\A^m$.
The conjugate module $\H_\infty^\3$ gives a pre-Morita equivalence
between $\A$ and $\CDA \isom q M_m(\A) q$.

Theorem~\ref{th:nice-product} tells us that we may form the spectral
triple $(\A, \H \ox_\A \H_\infty^\3, \Dhat) = (\A, \H^m q, \Dhat)$,
which satisfies the first-order condition, as well as Conditions
\ref{cn:qc-infty}, \ref{cn:spec-dim} and~\ref{cn:finite}.

Observe that Proposition~\ref{pr:summa} implies that
$$
\Dhatreg^{-p} = q^\op (\Dreg^{-p} \ox 1_m) q^\op + B,
\word{where}  B \in \L^{1,\infty}_0(\H^m q).
$$
To see that $B$ has vanishing Dixmier trace, we write
$\D'_m = q^\op (\D \ox 1_m) q^\op$, so that $\Dhat = \D'_m + \Ahat$
with $\Ahat$ bounded, by Lemma~\ref{lm:big-dee}. Using
\eqref{eq:inverse-perturbed}, we see that
$(i + \Dhat)^{-1} \equiv (i + \D'_m)^{-1} \bmod \Zz_{p/2}$, and
thus $\Dhatreg^{-1} \equiv \reg{\D'_m}^{-1} \bmod \Zz_{p/2}$.

The operator trace over $\H^m q$ is
$\Tr_\H \ox \tr_m(q^\op(\cdot)q^\op)$ with $\tr_m$ denoting a matrix
trace. The left action of $\CDA$ commutes with $q^\op$; thus, if $w$
is an even element of~$\CDA$, we get an equality of Dixmier traces:
\begin{equation}
\Trw^{\H^m q} \bigl( w\,\Dhatreg^{-p} \bigr)
= \Trw^{\H^m q} \bigl( w q^{\op} \,\Dreg^{-p} \bigr)
= \Trw^\H \bigl( w\tr_m(q^\op)\,\Dreg^{-p} \bigr).
\label{eq:non-comm-int} 
\end{equation}

We must now show that the new spectral triple satisfies
Condition~\ref{cn:Riemann}. The spin$^c$ condition, namely that
$\H_\infty$ is a pre-Morita equivalence bimodule between $\CDA$
and~$\A$, shows that there are finitely many vectors
$\xi_j,\eta_j \in \H_\infty$ such that 
\begin{equation}
\Theta_{\xi_1,\eta_1} +\cdots+ \Theta_{\xi_m,\eta_m} = 1 \in \CDA.
\label{eq:cyclic-1} 
\end{equation}
Consider the vector $\Phi \in \H \ox_\A \H_\infty^\3$ defined by
$$
\Phi := \xi_1 \ox \eta_1^\3 +\cdots+ \xi_m \ox \eta_m^\3.
$$
We claim that $\Phi$ is an algebraically cyclic vector for~$\CDA$,
and that the vector state
$\sg_\Phi \: w \mapsto \scal<\Phi|w\Phi>$ is of the form
$$
\sg_\Phi(w) = \Tr_\om^{\H^mq}(wz\Dhatreg^{-p})
$$
for a central element $z \in \CDA$.

Under the isomorphism
$\La \: \Theta_{\xi,\eta} \mapsto \xi \ox \eta^\3
: \CDA \to \H_\infty \ox_\A \H_\infty^\3$, the vector $\Phi$ is just 
the image $\La(1)$ of the unit element of~$\CDA$. Note that 
$$
1 = \sum_k \Theta_{\xi_k,\eta_k} = 1^2
= \sum_{j,k} \Theta_{\xi_k \pairR(\eta_k|\xi_j)_\A, \eta_j}
= \sum_{j,k} \Theta_{\xi_k, \eta_j \pairR(\xi_j|\eta_k)_\A}
$$
so that $\eta_k = \sum_j \eta_j \pairR(\xi_j|\eta_k)_\A$ for each~$k$.
Moreover, if $w = \sum_{i,k} \Theta_{\xi_i a_i, \eta_k b_k} \in \CDA$,
then
\begin{align}
w\Phi 
&= \sum_{i,j,k} \xi_i a_i\,\pairR(\eta_k b_k|\xi_j)_\A \ox \eta_j^\3
= \sum_{i,j,k} \xi_i a_i b_k^*\,\pairR(\eta_k|\xi_j)_\A \ox \eta_j^\3
\nonumber \\
&= \sum_{i,k} \xi_i a_i b_k^* \ox \eta_k^\3
= \sum_{i,k} \xi_i a_i \ox (\eta_k b_k)^\3 = \La(w).
\label{eq:both-sides} 
\end{align}
Thus $w \mapsto w\Phi = \La(w)$ maps $\CDA$ \textit{onto} 
$\H_\infty \ox_\A \H_\infty^\3$, so that $\Phi$ is algebraically 
cyclic; and it is separating as well, since the pre-Morita 
equivalence implies that $\La$ is bijective.

Using Proposition \ref{prop:spin-c-works}, the scalar product on the
Hilbert space $\H \ox_\A \H_\infty^\3$ is given on the dense subspace
$\H_\infty \ox_\A \H_\infty^\3$ by
\begin{align}
\scal< \xi \ox \eta^\3 | \zeta \ox \rho^\3 >
&:= \bigl< \xi\, \pairL_\A(\eta^\3 | \rho^\3) \bigm| \zeta \bigr>
= \bigl< \xi\, \pairR(\eta|\rho)_\A \bigm| \zeta \bigr>
\nonumber \\
&= \Trw\bigl( \pairR(\xi {\pairR(\eta|\rho)_\A} | \zeta)_\A
\,\Dreg^{-p} \bigr)
= \Trw\bigl( \pairR(\rho|\eta)_\A \pairR(\xi|\zeta)_\A \,\Dreg^{-p}
\bigr),
\label{eq:scalar-1} 
\end{align}
for $\xi,\eta,\zeta,\rho \in \H_\infty$.

\begin{lemma} 
\label{lm:affair-of-state}
Evaluation of the vector state $\sg_\Phi$ on the operator
$\Theta_{\rho,\tau} \ox 1$ in $\CDA \ox 1$ (acting on
$\H \ox \H_\infty^\3$) yields
$$
\sg_\Phi(\Theta_{\rho,\tau} \ox 1)
\equiv \scal<\Phi | (\Theta_{\rho,\tau} \ox 1) \Phi>
= \Trw^\H\bigl( \pairR(\tau|\rho)_\A \,\Dreg^{-p} \bigr).
$$
Moreover, $\sg_\Phi$ is tracial on~$\CDA$.
\end{lemma}

\begin{proof}
We first recall from Proposition \ref{prop:spin-c-works} that we can
use the right $\A$-valued inner product on $\H_\infty$ and $\psi_\om$
to compute the scalar product. The evaluation proceeds as follows, for
$\rho,\tau \in \H_\infty$:
\begin{align*}
\scal<\Phi | (\Theta_{\rho,\tau} \ox 1) \Phi>
&= \sum_{j,k} 
\scal<\xi_j \ox \eta_j^\3 | \Theta_{\rho,\tau} \,\xi_k \ox \eta_k^\3>
= \sum_{j,k} \bigl< \xi_j \ox \eta_j^\3 \bigm|
\rho\, \pairR(\tau|\xi_k)_\A \ox \eta_k^\3 \bigr>
\\
&= \sum_{j,k} \bigl< \xi_j \bigm| \rho\, \pairR(\tau|\xi_k)_\A \,
\pairR(\eta_k|\eta_j)_\A \bigr>
= \sum_{j,k} \bigl< \xi_j \bigm| \rho\, 
\pairR(\tau | \Theta_{\xi_k,\eta_k} \eta_j)_\A \bigr>
\\
&= \sum_j \bigl< \xi_j \bigm| \rho\, \pairR(\tau|\eta_j)_\A \bigr>
= \sum_j \Trw\bigl( \pairR(\xi_j | \rho\,{\pairR(\tau|\eta_j)_\A})_\A
\,\Dreg^{-p} \bigr)
\\
&= \sum_j \Trw\bigl( 
\pairR(\xi_j\, {\pairR(\eta_j|\tau)_\A} | \rho)_\A \,\Dreg^{-p} \bigr)
\\
&= 
\sum_j \Trw\bigl( \pairR(\Theta_{\xi_j,\eta_j}\tau | \rho)_\A
\,\Dreg^{-p} \bigr)
= \Trw\bigl( \pairR(\tau|\rho)_\A \,\Dreg^{-p} \bigr).
\end{align*}
Here we have used Condition~\ref{cn:finite} for the spectral triple
$(\A,\H,\D)$ and the adjointability of the right action of~$\A$.

We can now suppress the tensor factor ${} \ox 1$ in the notation for
the action of $\CDA$ on $\H \ox \H_\infty^\3$. The tracial property of
$\sg_\Phi$ on~$\CDA$ follows at once, since $\CDA$ is spanned by 
finite-rank endomorphisms:
\begin{align*}
\scal<\Phi | \Theta_{\rho,\tau} \Theta_{\al,\bt} \Phi>
&= \scal<\Phi | \Theta_{\rho \pairR(\tau|\al)_\A, \bt} \Phi>
= \Trw\bigl( \pairR(\bt | \rho\,{\pairR(\tau|\al)_\A} )_\A
\,\Dreg^{-p} \bigr)
\\
&= \Trw\bigl( \pairR(\bt|\rho)_\A \, \pairR(\tau|\al)_\A
\,\Dreg^{-p} \bigr)
= \Trw\bigl( \pairR(\tau\,{\pairR(\rho|\bt)_\A} | \al)_\A
\,\Dreg^{-p} \bigr)
\\
&= \Trw\bigl( \pairR(\tau | \al\,{\pairR(\bt|\rho)_\A} )_\A
\,\Dreg^{-p} \bigr)
= \Trw\bigl(\pairR(\tau | \Theta_{\al,\bt}\rho)_\A \,\Dreg^{-p}\bigr)
\\
&= \scal<\Phi | \Theta_{\al,\bt} \Theta_{\rho,\tau} \Phi>.
\tag*{\qedhere}
\end{align*}
\end{proof}

To complete the Riemannian requirements, we need a left-module 
$\CDA$-valued pairing on $\H_\infty \ox \H_\infty^\3$. The obvious 
choice is
$$
\pairL_\Cc({\La(u)}|{\La(v)}) := uv^*,  \word{for all} u,v \in \CDA.
$$
With this choice, $\pairL_\Cc(\Phi|\Phi) = 1$ is certainly central
and strictly positive.

To show that this choice of inner product satisfies the requirements
of Conditions \ref{cn:Riemann} and \ref{cn:finite}, we emulate the
identification of the scalar product on the Hilbert space
$\H \ox_\A \H_\infty^\3$ in~\eqref{eq:scalar-1}. On the dense subspace
$\H_\infty \ox_\A \H_\infty^\3$ we get
\begin{align}
\scal< \xi \ox \eta^\3 | \zeta \ox \rho^\3 >
&:= \bigl< \xi\, \pairL_\A(\eta^\3 | \rho^\3) \bigm| \zeta \bigr>
= \bigl< \xi\, \pairR(\eta|\rho)_\A \bigm| \zeta \bigr>
\nonumber\\
&= \Trw\bigl( \pairL_\Cc(\zeta|\xi {\pairR(\eta|\rho)_\A} )
\,\Dreg^{-p} \bigr)
= \Trw\bigl( \pairL_\Cc(\zeta| \pairL_\Cc(\xi|\eta)\rho) \,\Dreg^{-p}
\bigr)
\nonumber \\
&= \Trw\bigl( \pairL_\Cc(\zeta| \rho) \pairL_\Cc(\eta|\xi)\,\Dreg^{-p}
\bigr)
= \Trw\bigl( \Theta_{\zeta, \rho} \Theta_{\xi,\eta}^*\,\Dreg^{-p}
\bigr),
\label{eq:scalar-2} 
\end{align}
for $\xi,\eta,\zeta,\rho \in \H_\infty$. Thus the scalar product is
the composition of our chosen left $\CDA$-valued inner product and
$\psi_\om$. This gives us the required positivity and faithfulness as
well, since these properties hold for the spin$^c$ manifold
$(\A,\H,\D)$. Also $\H_\infty \ox_\A \H_\infty^\3 \isom \H_\infty^m q$
is finite projective as a left $\A$-module, since $\H_\infty$ is
finite projective by Proposition \ref{prop:spin-c-works}.

The Tomita conjugation $J := J_\Phi$ of~\eqref{eq:Tomita}
satisfies $J(T\Phi) = T^*\Phi$ for $T \in \CDA''$. 
Hence $J$ maps $\H_\infty \ox_\A \H_\infty^\3 = \La(\CDA)$ onto
itself. Indeed, from \eqref{eq:both-sides} one sees that $J$ is 
given on that subspace by $\xi \ox \eta^\3 \mapsto \eta \ox \xi^\3$.

In the even-dimensional case, the spin$^c$ condition gives
$$
\pi_{\Dhat}(\cc) = \pi_\D(\cc) \ox 1 = C \ox 1,
$$
so the Hochschild $p$-cycle is the same for both spectral triples. We 
can abbreviate $\Chat := \pi_{\Dhat}(\cc)$.

We claim that
$\eps := \pi_{\Dhat}(\cc)\, J \pi_{\Dhat}(\cc)J = \Chat J\Chat J$
anticommutes with~$\Dhat$. 

To see that, we use the standard isomorphism
$\psi \: q\A^m \to q(\C^m \ox \A)$ presenting elements of $q\A^m$ as
column vectors. Notice that
$(q(\C^m \ox \A))^\3 \isom (\A \ox \C^m)q$ with row vectors on the
right hand side. Using the standard bases of column vectors $u_i$ and
row vectors $v_i = u_i^T$ for~$\C^m$, this isomorphism is given by
$(q\sum_i u_i\ox a_i)^\3 \mapsto \sum_i (a_i^* \ox v_i)q$.

Since $\A$ and $q = 1_{\H_\infty}$ commute with the grading
$C = \pi_\D(\cc)$ of $\H_\infty \isom q\A^m$, there is a 
$\Z_2$-grading $G$ of $\C^m$ such that 
$\psi \circ C = q (1_\A \ox G) \circ \psi$. Using the 
identifications
$$
\H_\infty \ox_\A (\H_\infty)^\3 \isom (\H_\infty \ox \C^m) q
\isom \H_\infty \ox_\A \psi(\H_\infty)^\3
$$
and writing $\Phi = \sum_j \xi_j \ox_\A \eta_j^\3
= \sum_{i,j} \xi_j(\eta_j^i)^* \ox v_i$, we get
\begin{align*}
J\Chat J\Phi &= \sum_j \xi_j \ox_\A (C\eta_j)^\3 
= \sum_j \xi_j \ox_\A \sum_i (G u_i \ox \eta^i_j)^\3
\\
&= \sum_{i,j} \xi_j \ox_\A ((\eta_j^i)^* \ox v_iG)
= \sum_{i,j} \xi_j(\eta_j^i)^* \ox v_iG.
\end{align*}
The constructions in the proof of Lemma~\ref{lm:big-dee} show that 
$\Dhat$, as an operator on $(\H_\infty \ox \C^m)q$, may be written as
$\Dhat = q^\op(\D \ox 1_m)q^\op + q^\op(M \ox 1_m)q^\op$, where $M$ 
is a right $\A$-linear map on~$\H_\infty$. With that, we compute
$$
\Dhat J \Chat J \Phi
= \sum_{i,j} \D(\xi_j(\eta_j^i)^*) \ox v_iG 
+ \sum_j M\xi_j \ox (C\eta_j)^\3
= J \Chat J \Dhat \Phi.
$$
For a general element $S\Phi$ of $\H_\infty \ox_\A \H_\infty^\3$ 
---using \eqref{eq:both-sides} directly when $S$ is of finite rank---
we may simply replace $\xi_j \ox \eta_j^\3$ by $S(\xi_j \ox \eta_j^\3)$
in the above computations. Thus we find that
$J \Chat J \Dhat = \Dhat J \Chat J$.

Similarly, using $C\D = -\D C$, we find that
$\Chat\Dhat = -\Dhat\Chat$. For $\eps := \Chat J \Chat J$ we arrive at
$\Dhat \eps = -\eps \Dhat$.

Moreover, since $\eps\Phi = \Chat^2\Phi = \Phi$,
Condition~\ref{cn:Riemann} is established in the even case.

In the odd-dimensional case, we first use Appendix~\ref{app:KK-odd} to
replace our odd spectral triple with an even $\Z_2$-graded spectral
triple:
\begin{equation}
\biggl( \A \ox \A^\op \ox \Cliff_1, \H \ox \C^2,
\D' = \twobytwo{\D}{0}{0}{-\D}, \Ga = \twobytwo{0}{1}{1}{0},
\cc \biggr).
\label{eq:even-if-its-odd} 
\end{equation}
Also we take $\H_\infty\ox \C^2$ as a graded right module over
$\A \ox \Cliff_1$. Then the Kasparov product over $A \ox \Cliff_1$ of
the spectral triple \eqref{eq:even-if-its-odd} and 
$(\H_\infty \ox \C^2)^\3$ can be computed similarly to the even case,
and is represented by
$$
\biggl( \A \ox \CDA^\op, \H_\infty \ox_\A \H_\infty^\3 \ox \C^2,
\Dhat' = \twobytwo{\Dhat}{0}{0}{-\Dhat}, \Ga = \twobytwo{0}{1}{1}{0},
\cc \biggr).
$$
The discussion of summability is as in the even case, and we find that 
$\Phi \oplus \Phi$ is cyclic and separating, where
$\Phi := \Lambda(1_\Cc)$ as in the even case. Then
$\Chat' \equiv \pi_{\Dhat'}(\cc) = \twobytwo{\Chat}{0}{0}{-\Chat}$ and
the grading $\eps' := \Chat'(J \oplus J)\Chat'(J \oplus J)$ satisfy
the Riemannian conditions.

Observe that in the odd case, the left action of~$\CDA$ is given by
the algebra $\Cc_{\Dhat'}(\A) = (\CDA \ox 1) \oplus (\CDA \ox 1)$.

\vspace{6pt}

Now assume that the class $\mu \in KK^p(A\ox A^\op,\C)$ of
$(\A,\H,\D)$ satisfies the spin$^c$ version of Poincar\'e duality,
Condition \ref{cn:spinc-pdual}. We recall that the Morita equivalence
bimodule $(\H_\infty \ox \C^2_p)^\3$ defines a class $\sg$ in
$KK(C^\op, A^\op \ox \Cliff_1^p)$. Taking the Kasparov product with
the class $\sg$ gives two isomorphisms
\begin{align*}
- \ox_{C^\op} \sg
&: KK^j(\C, C^\op) \to KK^j(\C, A^\op \ox \Cliff_1^p),  \word{and}
\\
\sg \ox_{A^\op\ox\Cliff_1^p} -
&
: KK^j(A^\op \ox \Cliff_1^p,\C) \to KK^j(C^\op, \C).
\end{align*}
Then, combined with $\mu \in KK(A \ox A^\op \ox \Cliff_1^p, \C)$, we 
get another two isomorphisms
\begin{align*}
- \ox_{C^\op} \sg \ox_{A^\op\ox\Cliff_1^p} \mu
&: KK^j(\C, C^\op) \to KK^j(A, \C),  \word{and}
\\
\sg \ox_{A^\op\ox\Cliff_1^p} - \ox_A \mu
&: KK^j(\C, A) \to 
KK^j(C^\op, \C)
\end{align*}
are isomorphisms. The second follows because 
$$
(\sg \ox_{A^\op\ox\Cliff_1^p} -) \ox_A \mu
= \sg \ox_{A^\op\ox\Cliff_1^p} (- \ox_A \mu)
$$
by associativity of the Kasparov product, and $- \ox_A \mu$ gives an
isomorphism from $KK^j(\C, A) = K_j(A)$ to
$K^{j+p}(A^\op) = KK^j(A^\op \ox \Cliff_1^p, \C)$. Finally, it can be
explicitly checked by unpacking modules isomorphisms as in
Proposition~\ref{pr:married-triple} that
$$
- \ox_A (\sg \ox_{A^\op\ox\Cliff_1^p} \mu)
= \sg \ox_{A^\op\ox\Cliff_1^p} ( - \ox_A \mu).
$$
This shows that the class $\la := \sg \ox_{A^\op\ox\Cliff_1^p}\mu$
satisfies the Riemannian version of Poincar\'e duality,
Condition \ref{cn:riem-pdual}.
This completes the proof of Theorem~\ref{th:spin-riem}.
\qed

\subsection{Proof of Theorem \ref{th:riem-spin}} 

Our starting point here is a noncommutative oriented Riemannian
manifold $(\A,\H,\D,\cc,\Phi)$ and a pre-Morita equivalence bimodule
$\E$ between $\CDA$ and~$\A$ (even case); or between $\CDA^+$ and~$\A$
(odd case).

Proposition~\ref{pr:fun-class} yields an unbounded Kasparov module for 
the $\Z_2$-graded algebra $\A \ox \CDA^\op$.

The composition of pre-Morita equivalences is again a pre-Morita
equivalence, by a variant of \cite[Prop.~3.16]{RaeburnW}
or~\cite[Prop.~4.5]{Lance}. Since $\H_\infty$ is a pre-Morita
equivalence from $\Cc = \CDA$ to itself, the bimodule
$\H_\infty \ox_\Cc \E$ provides a pre-Morita equivalence between $\Cc$
and~$\A$ in the even case. Indeed, since $\H_\infty$ is free of rank
one over~$\Cc$, it is clear that $\H_\infty \ox_\Cc \E \isom \E$ as
$\Cc$-$\A$-bimodules; the isomorphism is given by
$\rho(\Phi \ox e) := e$.

If $p$ is even, the grading operator $\eps$ on~$\H$ need not extend 
to a well-defined grading operator on $\H \ox_\Cc \E$; indeed, if
$\E \isom \Cc^n q$, then $\H \ox_\Cc \E \isom \H^n q$ but
$q \in M_n(\CDA)$ need not be $\eps$-even. Instead, we must use
$\Chat := \pi_{\Dhat}(\cc) = C \ox 1$ and recall that $C$ anticommutes
with~$\D$ since $p$ is even.

Now we again use
$\Dhat = q^\op(\D \ox 1_m)q^\op + q^\op(M \ox 1_m)q^\op$ as an
operator on $(\H_\infty \ox \C^m)q$, where $M$ is a right $\A$-linear
operator on $\H_\infty$. With this description, it is straightforward
to check that $\Dhat \Chat = - \Chat \Dhat$. This proves the
orientation and all of the spin$^c$ condition in the even case.

To complete the proof of finiteness and absolute continuity, we must
display a left $\CDA$-valued inner product on
$\H_\infty \ox_\Cc \E \isom \E$ which captures the scalar product. We
define
$$
\pairL_\Cc(\Phi\ox e_1|\Phi\ox e_2)
:= \pairL_\Cc(\Phi|\Phi)\, \Theta_{e_1,e_2}
= \Theta_{e_1,e_2}\, \pairL_\Cc(\Phi|\Phi), \word{for} e_1,e_2 \in \E,
$$
using on the right hand side the given $\CDA$-valued inner product
on $\H_\infty$ and $\Theta_{e_1,e_2} \in \CDA$. Since
$z = \pairL_\Cc(\Phi|\Phi)$ is central and strictly positive, we see
that the new inner product is well defined. To show that this inner
product captures the scalar product on
$\H_\infty \ox_\Cc \E \isom \E$, we compute
\begin{align*}
\scal< \Phi \ox e_1 | \Phi \ox e_2 >_{\H\ox_\Cc\E}
&= \scal< \Phi | \Phi \Theta_{e_1,e_2} >_\H
= \scal< J(\Phi\Theta_{e_1,e_2}) | J\Phi >
\quad\text{as $J$ is anti-unitary}
\\
&= \scal< \Theta_{e_2,e_1}\Phi | J\Phi >
\word{since}  w\Phi = wJ\Phi = Jw^{*\op}\Phi = J(\Phi w^*)
\\
&= \scal< \Theta_{e_2,e_1}\Phi | \Phi >
= \psi_\om \bigl( \pairL_\Cc(\Phi|\Theta_{e_2,e_1}\Phi) \bigr)
\\
&= \psi_\om \bigl( \pairL_\Cc(\Phi|\Phi)\,\Theta_{e_1,e_2} \bigr)
= \psi_\om \bigl( \pairL_\Cc(\Phi\ox e_2|\Phi\ox e_1) \bigr).
\end{align*}
This demonstrates both parts of Condition \ref{cn:finite}.

As in the proof of Theorem~\ref{th:spin-riem}, $\Dhat$ satisfies the
regularity condition for the left action of~$\A$. So in the even case,
we are done.

If $p$ is odd, we take $\Cc := \CDA^+$ as defined in 
Corollary~\ref{cr:odd-big-algebra}. We now write
$\H = \H^+ \oplus \H^-$ where the splitting of~$\H$ is into
$(\pm 1)$-eigenspaces for $C = \pi_\D(\cc)$. We use the fact that
$C \ox 1$ acts as the identity on $\H^+ \ox_\Cc \E$, and
$\eps\,C = -C\,\eps$ in odd dimensions, to deduce that
$$
\eps \ox 1 \word{acts on} \H \ox_\Cc \E =
\begin{pmatrix} \H^+ \ox_\Cc \E \\[\jot] \H^- \ox_\Cc \E \end{pmatrix}
\word{as}  \twobytwo{0}{1}{1}{0}.
$$
Thus $\eps\ox 1$ provides  an isomorphism from $\H^+ \ox_\Cc \E$ to 
$\H^- \ox_\Cc \E$, and conjugation by $\eps$ carries $\CDA^\pm$ onto 
$\CDA^\mp$, by Corollary~\ref{cr:odd-big-algebra}. These facts give us
a Hilbert space \emph{and} left $\CDA$-module isomorphism
\begin{equation}
\H \ox_\Cc \E \isom \H^+ \ox_\Cc \E \ox \C^2.
\label{eq:gotta-c2} 
\end{equation}
The copy of $\C^2$ in Equation \eqref{eq:gotta-c2} is graded by the
action of $C \ox 1$ as $\twobytwo{1}{0}{0}{-1}$. The two-dimensional
algebra $\C\langle 1, \eps\rangle \isom \Cliff_1$ acts on~$\C^2$ with
$\eps$ acting as $\twobytwo{0}{1}{1}{0}$.

The operator $\Dhat$ anticommutes with $\eps \ox 1$ and commutes with
$C \ox 1$, so in the two-by-two matrix picture, it must be of the form
$\Dhat =: \twobytwo{\Dhat'}{0}{0}{-\Dhat'}$ for a selfadjoint operator
$\Dhat'$. The relations satisfied by $\Dhat$ now imply that we obtain
an unbounded even Kasparov module for the algebra $\A \ox \Cliff_1$.

Now by the discussion in the Appendix, $(\A, \H^+ \ox_\Cc \E, \Dhat')$
is an odd spectral triple. Similarly to the even case, we find that
$\H^+ \ox_\Cc \E \isom \E$ via 
$P_+ w\Phi \ox e \mapsto P_+ w P_+ e$, since $P_+ = \half(1 + C)$ is
central and $\E$ is a pre-Morita equivalence bimodule from $\CDA^+$ to
$\A$. Hence the spectral triple $(\A, \H^+ \ox_\Cc \E, \Dhat')$
satisfies the spin$^c$ condition.

Using $\la \in KK(A \ox C^\op, \C)$ and
$\tau = [\E \ox \C^2_p] \in KK(A^\op \ox \Cliff_1^p, C^\op)$, we 
compose the isomorphisms $- \ox_A \la : KK^j(\C,A) \to KK^j(C^\op,\C)$
and 
$\tau \ox_{C^\op} - : KK^j(C^\op,\C) \to KK^j(A^\op\ox\Cliff_1^p,\C)$
and thereby get an isomorphism
$$
- \ox_A (\tau \ox_{C^\op} \la) = \tau \ox_{C^\op} (- \ox_A \la) 
: KK^j(\C, A) \to KK^j(A^\op \ox \Cliff_1^p, \C) \isom KK^{j+p}(A,\C).
$$
Thus the class $\mu := \tau \ox_{C^\op} \la$ satisfies the spin$^c$
Poincar\'e duality condition.
\qed

\appendix

\addtocontents{toc}{\vspace{-6pt}}

\section{Appendix on odd $KK$-classes}
\label{app:KK-odd}

Suppose that $(\A,\H,\D)$ is an odd spectral triple for the ungraded
algebra $\A$. This defines a class in
$KK^1(A, \C) \isom KK^0(A \ox \Cliff_1, \C)$, and we present a
$\Z_2$-graded even spectral triple for $\A \ox \Cliff_1$ representing
$(\A,\H,\D)$ in the even $KK$-group.

The representative we use is
$$
\biggl( \A \ox \Cliff_1, \H \ox \C^2, \D' = \twobytwo{\D}{0}{0}{-\D},
\Ga' = \twobytwo{0}{1}{1}{0} \biggr),
$$
where $\Cliff_1$ is generated by $\twobytwo{0}{-i}{i}{0}$. In
\cite[Prop.~IV.A.13]{Book}, Connes employs a different representative,
namely
$$
\biggl( \A \ox \Cliff_1, \H \ox \C^2, \D'' = \twobytwo{0}{-i\D}{i\D}{0},
\Ga'' = \twobytwo{1}{0}{0}{-1} \biggr),
$$
with $\Cliff_1$ generated by $\twobytwo{0}{1}{1}{0}$.

These two representatives define the same $KK$-class. To see this,
first employ the unitary equivalence defined by
$U = \twobytwo{1}{0}{0}{i}$, which conjugates Connes' representative
to
$$
\biggl( \A \ox \Cliff_1, \H \ox \C^2, \D''' = \twobytwo{0}{-\D}{-\D}{0},
\Ga'' = \twobytwo{1}{0}{0}{-1} \biggr),
$$
with $\Cliff_1$ now generated by $\twobytwo{0}{-i}{i}{0}$. Next, we
employ the homotopy
$$
\biggl( \A \ox \Cliff_1, \H \ox \C^2, 
\D_t = \twobytwo{\D\sin t}{-\D\cos t}{-\D\cos t}{-\D\sin t}, 
\Ga_t = \twobytwo{\cos t}{\sin t}{\sin t}{-\cos t} \biggr), \quad
0 \leq t \leq \frac{\pi}{2}.
$$
This homotopy of $\Z_2$-graded spectral triples takes us from
$\D''',\,\Ga''$ to $\D',\,\Ga'$, and so the equality of the
$KK$-classes is established. The same argument can be carried through
unchanged for the associated Kasparov modules defined by applying the
real function $x \mapsto x(1 + x^2)^{-1/2}$ to the various operators
$\D$, $\D'$, $\D''$, $\D'''$.

Now given an even ($\Z_2$-graded) Kasparov module for the
$\Z_2$-graded algebra $A \ox \Cliff_1$ and~$\C$, we can take a
representative of the form
\begin{equation}
\biggl( A \ox \Cliff_1, \H \ox \C^2, \twobytwo{F}{0}{0}{-F},
\Ga' = \twobytwo{0}{1}{1}{0} \biggr),
\label{eq:kas-mod} 
\end{equation}
by taking $\Cliff_1$ to be generated by $\twobytwo{0}{-i}{i}{0}$. Then
an odd Kasparov module for the ungraded algebra $A$, representing the
class in $KK^1(A,\C)$ obtained from the class of the Kasparov module
\eqref{eq:kas-mod} via the isomorphism
$KK^0(A \ox \Cliff_1, \C) \isom KK^1(A, \C)$, is given by $(A,\H,F)$.

Replacing the bounded operator $F$ by an unbounded operator $\D$ does
not alter the discussion.


\begin{thebibliography}{40}

\bibitem{BaajJ}
S. Baaj and P. Julg,
``Th\'eorie bivariante de Kasparov et op\'erateurs non born\'ees dans
les $C^*$-modules hilbertiens'',
C. R. Acad. Sci. Paris \textbf{296} (1983), 875--878.

\bibitem{BerlineGV}
N. Berline, E. Getzler and M. Vergne,
\textit{Heat Kernels and Dirac Operators},
Springer, Berlin, 1992.

\bibitem{Blackadar}
B. Blackadar,
\textit{$K$-Theory for Operator Algebras}, 2nd edition,
Cambridge University Press, Cambridge, 1998.

\bibitem{BoeijinkS}
J. Boeijink and W. D. van Suijlekom,
``The noncommutative geometry of Yang--Mills fields'',
J. Geom. Phys. \textbf{61} (2011), 1122--1134.

\bibitem{BratteliRoI}
O. Bratteli and D. W. Robinson,
\textit{Operator Algebras and Quantum Statistical Mechanics 1},
2nd edition, Springer, New York, 1987.

\bibitem{Chachich}
B. \'Ca\'ci\'c,
``A reconstruction theorem for almost-commutative spectral triples'',
Lett. Math. Phys. \textbf{100} (2012), in press.

\bibitem{CareyGRS1}
A. L. Carey, V. Gayral, A. Rennie, F. Sukochev,
``Integration on locally compact noncommutative spaces'',
math.OA/0912.2817.

\bibitem{CareyP1}
A. L. Carey and J. Phillips, 
``Unbounded Fredholm modules and spectral flow'',
Can. J. Math. \textbf{50} (1998), 673--718.

\bibitem{CareyPRSHoch}
A. L. Carey, J. Phillips, A. Rennie and F. A. Sukochev,
``The Hochschild class of the Chern character for semifinite spectral
triples'',
J. Funct. Anal. \textbf{213} (2004), 111--153.

\bibitem{CareyRSS}
A. L. Carey,  A. Rennie, F. A. Sukochev, A. Sedaev,
``The Dixmier trace and asymptotics of zeta functions'',
J. Funct. Anal. \textbf{249} (2007), 253--283.

\bibitem{CiprianiGS}
F. Cipriani, D. Guido, S. Scarlatti,
``A remark on trace properties of $K$-cycles'',
J. Oper. Theory \textbf{35} (1996), 179--189.

\bibitem{Book}
A. Connes,
\textit{Noncommutative Geometry},
Academic Press, London and San Diego, 1994.

\bibitem{ConnesGrav}
A. Connes,
``Gravity coupled with matter and foundation of noncommutative
geometry'',
Commun. Math. Phys. \textbf{182} (1996), 155--176.

\bibitem{ConnesRecon}
A. Connes,
``On the spectral characterization of manifolds'',
math.OA/0810.2088.

\bibitem{FroehlichGR}
J. Fr\"ohlich, O. Grandjean and A. Recknagel,
``Supersymmetric quantum theory and noncommutative geometry'',
Commun. Math. Phys. \textbf{203} (1999), 119--184.

\bibitem{Himalia}
V. Gayral, J. M. Gracia-Bond\'ia, B. Iochum, T. Sch\"ucker and
J.~C. V\'arilly,
``Moyal planes are spectral triples'',
Commun. Math. Phys. \textbf{246} (2004), 569--623.

\bibitem{Selene}
J. M. Gracia-Bond\'ia, F. Lizzi, G. Marmo and P. Vitale,
``Infinitely many star products to play with'',
J. High Energy Phys. \textbf{04} (2002) 026.

\bibitem{Polaris}
J. M. Gracia-Bond\'ia, J. C. V\'arilly and H. Figueroa,
\textit{Elements of Noncommutative Geometry},
Birkh\"auser, Boston, 2001.

\bibitem{Helemskii}
A. Ya. Helemskii,
\textit{The Homology of Banach and Topological Algebras},
Kluwer, Dordrecht, 1989.

\bibitem{KasparovTech}
G. G. Kasparov,
``The operator $K$-functor and extensions of $C^*$-algebras'',
Math. USSR Izv. \textbf{16} (1981), 513--572.

\bibitem{KasparovEqvar}
G. G. Kasparov,
``Equivariant $KK$-theory and the Novikov conjecture'',
Invent. Math. \textbf{91} (1988), 147--201.

\bibitem{Kosaki}
H. Kosaki,
\textit{Type III Factors and Index Theory},
RIM--GARC Lecture Notes \textbf{43},
Seoul National University, Seoul, 1998.

\bibitem{Krajewski}
T. Krajewski, ``Classification of finite spectral triples'',
J. Geom. Phys. \textbf{28} (1998), 1--30.

\bibitem{Kucerovsky}
D. Kucerovsky,
``The $KK$-product of unbounded modules'', 
$K$-Theory \textbf{11} (1997), 17--34.

\bibitem{Lance}
E. C. Lance,
\textit{Hilbert $C^*$-modules}, 
Cambridge University Press, Cambridge, 1995.

\bibitem{Loday}
J.-L. Loday,
\textit{Cyclic Homology},
2nd edition, Springer, Berlin, 1996.

\bibitem{LordThesis}
S. Lord,
``Riemannian noncommutative geometry'',
Ph.D. thesis, University of Adelaide, 2004.

\bibitem{LSReview}
S. Lord and F. A. Sukochev,
``Measure theory in noncommutative spaces'',
SIGMA \textbf{6} (2010), 072.

\bibitem{Mesland}
B. Mesland,
``Unbounded bivariant $K$-theory and correspondences in noncommutative
geometry'',
math.KT/0904.4383v4.

\bibitem{PaschkeS}
M. Paschke and A. Sitarz,
``Discrete spectral triples and their symmetries'',
J. Math. Phys. \textbf{39} (1998), 6191--6205.

\bibitem{Pedersen}
G. K. Pedersen,
\textit{$C^*$-algebras and Their Automorphism Groups},
Academic Press, London, 1979.

\bibitem{Plymen}
R. J. Plymen,
``Strong Morita equivalence, spinors and symplectic spinors'',
J. Operator Theory \textbf{16} (1986), 305--324.

\bibitem{RaeburnW}
I. Raeburn and D. P. Williams,
\textit{Morita Equivalence and Continuous-Trace $C^*$-algebras},
Mathematical Surveys and Monographs \textbf{60},
American Mathematical Society, Providence, RI, 1998.

\bibitem{RennieComm}
A. Rennie,
``Commutative geometries are spin manifolds'',
Rev. Math. Phys. \textbf{13} (2001), 409--464.

\bibitem{RennieSmooth}
A. Rennie,
``Smoothness and locality for nonunital spectral triples'',
K-Theory \textbf{28} (2003), 127--165.

\bibitem{Crux}
A. Rennie and J. C. V\'arilly,
``Reconstruction of manifolds in noncommutative geometry'',
math.OA/0610418.

\bibitem{RieffelInd}
M. A. Rieffel,
``Induced representations of $C^*$-algebras'',
Adv. Math. \textbf{13} (1974), 176--257.

\bibitem{Schweitzer}
L. B. Schweitzer,
``A short proof that $M_n(A)$ is local if $A$ is local and
Fr\'echet'',
Int. J. Math. \textbf{3} (1992), 581--589.

\bibitem{TakesakiTwo}
M. Takesaki,
\textit{Theory of Operator Algebras II},
Springer, Berlin, 2003.

\bibitem{Zhang}
D. Zhang,
``Projective Dirac operators, twisted K-theory and local index
formula'',
math.DG/1008.0707.

\end{thebibliography}
\end{document}